\documentclass[12pt]{amsart}
\usepackage{amssymb,amsthm}\allowdisplaybreaks[4]
\usepackage{url}
\usepackage{graphicx}   
\usepackage{verbatim}       
\usepackage{layout}
\usepackage{galois}
\usepackage{mathrsfs,amsfonts}
\usepackage{color}
\usepackage{bbm}
\usepackage{enumerate}
\usepackage[round]{natbib}
\newcommand{\ind}{1{\hskip -2.5 pt}\hbox{I}}

\usepackage{amsmath}
\numberwithin{equation}{section}

\usepackage{geometry}
\newtheorem{theorem}{Theorem}[section]
\newtheorem{lemma}[theorem]{Lemma}
\newtheorem{proposition}[theorem]{Proposition}

\newtheorem{definition}[theorem]{Definition}

\def\beqlb{\begin{eqnarray}}\def\eeqlb{\end{eqnarray}}
 \def\beqnn{\begin{eqnarray*}}\def\eeqnn{\end{eqnarray*}}

\def\be{{\beta}}

\def\<{\left<}\def\>{\right>}
\def\({\left(}\def\){\right)}
\def\[{\left[}\def\]{\right]}
 \def\red{\color{red}}
 
 \def\beqlb{\begin{eqnarray}}\def\eeqlb{\end{eqnarray}}
 \def\beqnn{\begin{eqnarray*}}\def\eeqnn{\end{eqnarray*}}
 \def\d{{\mbox{\rm d}}}
\def\ar{\!\!&}\def\nnm{\nonumber}\def\ccr{\nnm\\}

\begin{document}

\title{Mutually interacting superprocesses with migration}
\thanks{The research of L. Ji was supported in part by the fellowship of China Postoctoral Science Foundation  2020M68194; the research of H. Liu was supported in part by NSF of Hebei Province A2019205299, Hebei Education Department QN2019073, NSFC 11501164 and HNU L2019Z01; the research of J. Xiong  was supported in part by SUSTech fund Y01286120 and NSFC grants 61873325 and 11831010.}
\author{Lina Ji, Huili Liu and Jie Xiong}
\address{Lina Ji: Department of Mathematics, Southern University of Science and Technology}
\email{jiln@sustch.edu.cn}
\address{Huili Liu: School of Mathematical Sciences, Hebei Normal University,
Shijiazhuang, Hebei, China}
\email{liuhuili@hebtu.edu.cn}
\address{Jie Xiong: Department of Mathematics and National Center for Applied Mathematics (Shenzhen), Southern University of Science and Technology}
\email{xiongj@sustch.edu.cn} \subjclass[2010]{Primary 60J68; Secondary 60H15; 60K35}
\date{\today}
\keywords{Distribution-function-valued; Martingale problem; Superprocess; Stochastic partial differential equation; Well-posedness }
\begin{abstract}
A system of mutually interacting superprocesses with migration is constructed as the limit of a sequence of branching particle systems arising from population models. The uniqueness in law of the superprocesses is established using the pathwise uniqueness of a system of stochastic partial differential equations with  non-Lipschitz coefficients, which is satisfied  by the corresponding system of distribution-function-valued processes.

\end{abstract}
\maketitle \pagestyle{myheadings} \markboth{\textsc{superprocesses with one-way migration}} {\textsc{superprocesses with one-way migration}}
\section{Introduction}
Superprocesses, describing the evolution of large population undergoing random reproduction and spatial motion, were first constructed as high-density limits of branching particle systems by  \cite{W68}. The connection between superprocesses with stochastic evolution equations was investigated by \cite{D75}. Since then ample systematic research results have been published, see e.g., \cite{Dawson, Etheridge, Li2011}. The ones with immigration, 
 a class of generalizations of superprocesses, have also attracted the interest of many researchers.
We refer to  \cite{Li2011, LS1995, LXZ2010} and the references therein for immigration structure and related properties.
 Let $M_{F}\(\mathbbm{R}\)$ be the collection of all finite Borel measures on $\mathbbm{R}$. Denote by $C_{b}^k\(\mathbbm{R}\)$ (resp. $C_{0}^k\(\mathbbm{R}\)$) the collection of all bounded (resp. compactly-supported) continuous functions on $\mathbbm{R}$ with bounded derivatives up to $k$th order. We consider a continuous $M_{F}\(\mathbbm{R}\)$-valued process $\(\mu_t\)_{t\geq 0},$ solving the following martingale problem (MP): $\forall\,f\in C_{b}^2\(\mathbbm{R}\)$, the process
\begin{equation}\label{cmp1}
M_t^f=\< \mu_t,f\>-\< \mu_0,f\>-\int_0^t\(\frac12\< \mu_s,f^{''}\> +\<\kappa,f\>\)ds
\end{equation}
is a continuous martingale with quadratic variation process
\begin{equation}\label{cmp2}
\< M^f\>_t=\gamma\int_0^t\< \mu_s,f^2\> ds
\end{equation}
where $\gamma>0$ and $\kappa\in M_{F}\(\mathbbm{R}\)$.
The corresponding model is super-Brownian motion (SBM) when $\kappa=0$. The uniqueness in law of SBM can be obtained by its log-Laplace equation. \cite{X2013} studied the stochastic partial differential equation (SPDE) satisfied by the distribution-function-valued process, and 
approached the uniqueness of SBM
from a different point of view. Related work can also be found in \cite{DL2012AP, HLY}.

For  a finite measure $\kappa$, the corresponding model is superprocess with immigration, which was constructed in \cite{Li1992} through the cumulant semigroup, see also \cite{Li1996, LS1995}.
In case of $\kappa$ being interactive, i.e., $\kappa=\kappa\(\mu_s\)$,
the existence of solution to MP (\ref{cmp1}, \ref{cmp2}) has been verified in \cite{M92}, where the result
is applicable to the situation with interactive immigration, branching rate and spatial motion.
By the approach of pathwise uniqueness for SPDEs satisfied by the distribution-function-valued process, \cite{MX2015}
established the well-posedness of MP for superprocess with interactive immigration. See also \cite{XY2016} for the related work.

However, there exist some populations distributed in different colonies, such as the mutually catalytic branching model,
see \cite{DP1998AP, D2002AP, D2002EJP,M96, M1998}.
The evolution of  such a model can be depicted by interacting superprocesses. Some sudden event may induce mass migration and lead to an increment of population size in one colony while a decrement in the other. For instance, war makes large numbers of people move into neighboring country and radiation mutates normal cells and so on. Therefore, it is natural to study the superprocesses with
interactive migration
between different colonies.  In this paper we  consider a continuous $M_{F}\(\mathbbm{R}\)^2$-valued process $\(\mu_t^{\mathbbm{1}},\mu_t^{\mathbbm{2}}\)_{t\geq0}$, called as {\it mutually interacting superprocesses with migration}.
It solves the following 
MP: $\forall\ f, g\in C_{0}^2\(\mathbbm{R}\)$, 
the processes
\begin{equation}\label{mp1}
\begin{cases}
M_t^{f}=\langle\mu_t^{\mathbbm{1}},f\rangle-\langle\mu_0^{\mathbbm{1}},f\rangle-\frac12 \int_0^t\langle\mu_s^{\mathbbm{1}},f^{''}\rangle ds
-b_{\mathbbm{1}}\int_0^t\langle\mu_s^{\mathbbm{1}},f\rangle ds+\int_0^t\langle\mu_s^{\mathbbm{1}},\eta\(\cdot,\mu_s^{\mathbbm{1}},\mu_s^{\mathbbm{2}}\)
f\rangle ds\\
\hat{M}_t^{g}=\langle\mu_t^{\mathbbm{2}},g\rangle-\langle\mu_0^{\mathbbm{2}},g\rangle-\frac12 \int_0^t\langle\mu_s^{\mathbbm{2}},g^{''}\rangle ds
-b_{\mathbbm{2}}\int_0^t\langle\mu_s^{\mathbbm{2}},g\rangle ds-\langle\chi,g\rangle\int_0^t\langle\mu_s^{\mathbbm{1}}, \eta\(\cdot,\mu_s^{\mathbbm{1}},\mu_s^{\mathbbm{2}}\)\rangle ds\\
\end{cases}
\end{equation}
are two continuous martingales with quadratic variation and covariation processes
\begin{equation}\label{mp2}
\begin{cases}
\langle M^{f}\rangle_t=\gamma_{\mathbbm{1}}\int_0^t\langle\mu_s^{\mathbbm{1}},f^2\rangle ds\\
\langle{\hat{M}}^{g}\rangle_t=\gamma_{\mathbbm{2}}\int_0^t\langle \mu_s^{\mathbbm{2}},g^2\rangle ds\\
\langle{M^{f},\hat{M}}^{g}\rangle_t=0
\end{cases}
\end{equation}
where $\chi$ is a finite measure on $\mathbb{R}$; $\gamma_{\mathbbm{1}}$ and $\gamma_{\mathbbm{2}}$ are positive constants; the migration intensity $\eta\(\cdot,\cdot,\cdot\)\in C_b^{+}\(\mathbbm{R}\times M_F(\mathbbm{R})^2\)$ is a nonnegative bounded continuous functions on $\mathbbm{R} \times M_F(\mathbbm{R})^2$.

The {\it purpose} of this paper is to establish the well-posedness of MP  (\ref{mp1},\,\ref{mp2}), i.e.,
the existence and uniqueness in law of such a mutually interacting superprocesses with migration. The process is constructed as the limit of a sequence of branching particle systems. 
To obtain the uniqueness in law of the superprocesses, we demonstrate the pathwise uniqueness of the solution to a system of mutually interacting SPDEs with non-Lipschitz coefficients, which are satisfied by the corresponding distribution-function-valued processes. 
 As far as we know, this is the first attempt to discuss the well-posedness of mutually interactive superprocesses with  migration.

Now we present some notation. Let $D(\mathbbm{R}_+, M_F(\mathbbm{R})^2)$ (resp. $C(\mathbbm{R}_+, M_F(\mathbbm{R})^2)$) denote the space of c\'{a}dl\'{a}g (resp. continuous) paths from $\mathbbm{R}_+$ to $M_F(\mathbbm{R})^2$ furnished with the Skorokhod topology. Let $D(\mathbbm{R}_+, \mathbbm{R}^2)$ be the collection of  c\'{a}dl\'{a}g paths from $\mathbbm{R}_+$ to $\mathbbm{R}^2$. Recall that $C_b(\mathbbm{R})$ is the collection of all bounded continuous functions on $\mathbbm{R}.$ Let $C_{b,m}\(\mathbbm{R}\)$ be the subset of $C_b\(\mathbbm{R}\)$ consisting of nondecreasing bounded continuous functions on $\mathbbm{R}$. Write $\left<\mu,f\right>$ as the integral of $f \in C_b^2(\mathbbm{R})$ with respect to measure $\mu \in M_F(\mathbbm{R})$. For any $f, g \in C_b^2(\mathbbm{R}),$ define $\langle f, g \rangle_1 = \int_{\mathbbm{R}}f(x)g(x) d x.$
Let $H_0=L^2\(\mathbbm{R}\)$ be the Hilbert space consisting of all square-integrable functions with Hilbertian norm $\|\cdot\|_0$ given by $\|f\|_0^2=\int_{\mathbbm{R}}f^2\(x\)e^{-|x|}dx$
for any $f\in H_0$.
Set $v_i\(x\)=\nu_i\((-\infty,x]\)$
as the distribution functions of $\nu_i\in M_{F}\(\mathbbm{R}\)$ for any $x\in\mathbbm{R}$ and  $i=1,2$. Define distance $\rho$ on $M_{F}\(\mathbbm{R}\)$ by
\beqlb\label{rho}
\rho\(\nu_1,\nu_2\)=\int_{\mathbbm{R}}e^{-\mid x\mid }\mid v_1\(x\)-v_2\(x\)\mid dx.
\eeqlb
It is easy to see that, under metric distance $\rho,$ $M_{F}\(\mathbbm{R}\)$ is a Polish space whose topology coincides with that given by weak convergence of measures. Moreover, we always assume that all random variables in this paper are defined on the same filtered  probability space $(\Omega, \mathcal{F}, \mathcal{F}_t, \mathbbm{P}).$  Let $\mathbbm{E}$ be the corresponding expectation.

The rest of this paper is organized as follows. In Section \ref{se:model}, a sequence of branching particle systems arising from population models is introduced. In Section \ref{existence}, we show the existence of solution to MP (\ref{mp1},\,\ref{mp2}) by the  convergence of the branching particle systems. In Section \ref{uniqueness}, the equivalence between MP  (\ref{mp1},\,\ref{mp2}) and SPDEs satisfied by the distribution-function-valued processes is established and further we prove the pathwise uniqueness of the SPDEs by an extended Yamada-Watanabe argument. Throughout the paper, we let the letter K with or without subscripts
to denote constants whose exact value is unimportant and may change from line to line.

\section{ A related branching model with migration}\label{se:model}
There exists a population living in two colonies with labels $\{\mathbbm{1,\,2}\}$. Initially, each colony has $n$ particles, spatially distributed in $\mathbbm{R}$. Write $k\sim t$ to mean that the $k$th living particle at time $t$ in each colony. For any $s\geq t$, denote by $X_{k\sim t}^{n,\mathbbm{i}}(s)$
the corresponding particle's location
at time $s$
in colony $\mathbbm{i}$ with  $\mathbbm{i}\in\{\mathbbm{1,\,2}\}$ if it is alive up to time $s$.
The motions of the particles during their life times are modeled by independent Brownian motions.
For any $\mathbbm{i}\in\{\mathbbm{1,2}\}$, $k\sim t$, accompanying the corresponding particle with standard Brownian motion $\{B_{k\sim t}^{\mathbbm{i}}\(s\):s\geq 0\}$, we have
\begin{equation*}
X_{k\sim t}^{n,\mathbbm{i}}\(s\)=
X_{k\sim t}^{n,\mathbbm{i}}\(t\)+B_{k\sim t}^{\mathbbm{i}}\(s\)-B_{k\sim t}^{\mathbbm{i}}\(t\), \qquad \forall\ s\geq t.
\end{equation*}
If $s<t$, the same notation $X_{k\sim t}^{n,\mathbbm{i}}(s)$ represents the location of the corresponding particle's mother at time $s$.
Denote by $\lfloor x\rfloor$ the integer part of $x$. Let $h$ be a positive constant which is small enough.
For any exponential random variable $Y$ with parameter $\lambda$,
we have
\begin{equation*}
\lim_{\Delta\rightarrow 0+}h^{-1}\mathbbm{P}\(t\leq Y< t+\Delta\)=\lambda.
\end{equation*}

Now we are ready to introduce the branching particle systems. In colony $\mathbbm{1}$, there exist independent branching and emigration, and there are also independent branching and immigration in colony $\mathbbm{2}$. The branching mechanisms in two colonies are also independent. However, the emigration and immigration are interactive. 
The particles in colony $\mathbbm{1}$ can move to colony $\mathbbm{2}$ and the opposite direction is not allowed. During branching/emigration/immigration events, all the particles move according to independent Brownian motions.
\begin{itemize}
\item (Measure-valued process in colony $\mathbbm{i}$ with $\mathbbm{i} = \mathbbm{1}, \mathbbm{2}$) Let $\mu_t^{n,\mathbbm{i}}$ be the empirical distribution of particles living in colony $\mathbbm{i}$, i.e., for any $f\in C_{b}^{2}\(\mathbbm{R}\)$, we have
    $$\left<\mu_t^{n,\mathbbm{i}},f\right>=\frac{1}{n}\sum_{k\sim t}f\(X_{k\sim t}^{n,\mathbbm{i}}\(t\)\),$$
    where the sum $k\sim t$ includes all those particles  alive at $t$ in each colony.
\item (Branching in colony $\mathbbm{i}$ with $\mathbbm{i} = \mathbbm{1}, \mathbbm{2}$) Let $\lambda_{n,\mathbbm{i}}$ be the branching rate of particles in colony $\mathbbm{i}$ with $\frac{\lambda_{n,\mathbbm{i}}}{n}\rightarrow\lambda_{\mathbbm{i}}$ as $n\rightarrow\infty$.  At the time of a particle's death, it gives birth to a random number $\xi^{n, \mathbbm{i}}$ of offspring with $\mathbbm{E}\xi^{n, \mathbbm{i}}=1+\frac{\beta_{n,\mathbbm{i}}}{n}$, $\mathbbm{Var}\xi^{n, \mathbbm{i}}=\sigma^2_{n,\mathbbm{i}}$ satisfying $\beta_{n,\mathbbm{i}}\rightarrow\beta_{\mathbbm{i}}$ and $\sigma_{n,\mathbbm{i}}\rightarrow\sigma_{\mathbbm{i}}$ as $n\rightarrow\infty$. The offspring start moving from their mothers' locations.
\item (Emigration in colony $\mathbbm{1}$) For a particle alive at time $t$ and position $x$ in colony $\mathbbm{1}$, the conditional probability of emigration in the time interval $[t,t+h)$ is
\beqnn
1-\exp\left\{-\int_{t}^{t+h}\eta\(x, \mu_s^{n,\mathbbm{1}}, \mu_s^{n,\mathbbm{2}}\)ds\right\}\approx h\eta\(x, \mu_t^{n,\mathbbm{1}},\mu_t^{n,\mathbbm{2}}\),
\eeqnn
where $h$ is small enough and $\eta\(x,  \mu_s^{n,\mathbbm{1}},\mu_s^{n,\mathbbm{2}}\)\in C_b^+(\mathbbm{R}\times M_F(\mathbbm{R})^2)$ is the emigration intensity. 
\item (Immigration in colony $\mathbbm{2}$) The emigration at time $t$ in colony $\mathbbm{1}$ induces the immigration in colony $\mathbbm{2}$, with total immigrants $\langle \mu_t^{n,\mathbbm{1}}, \eta(\cdot,\mu_t^{n,\mathbbm{1}},\mu_t^{n,\mathbbm{2}})\rangle$. Therefore, the probability of immigration to colony $\mathbbm{2}$ in the time interval $[t, t + h)$ is
    \beqnn
    &\ &1-\exp\left\{-\int_{t}^{t+h}\langle \mu_s^{n,\mathbbm{1}}, \eta\(\cdot, \mu_s^{n,\mathbbm{1}},\mu_s^{n,\mathbbm{2}}\)\rangle d s\right\}\approx \langle \mu_t^{n,\mathbbm{1}}, \eta(\cdot, \mu_t^{n,\mathbbm{1}},\mu_t^{n,\mathbbm{2}})\rangle h.
    \eeqnn
    Moreover, we assume that the immigrants settle down in colony $\mathbbm{2}$ according to a finite measure $\chi$ on $\mathbbm{R}.$

\end{itemize}


\section{Existence of solution to the martingale problem}\label{existence}
In this section,  we study the convergence of a sequence of measure-valued processes arising as the empirical measures of the proposed branching particle systems in previous section, where the limit is a weak solution to MP (\ref{mp1},\,\ref{mp2}).

Given any $t>0$  and $\mathbbm{i}\in\{\mathbbm{1,2}\}$, denote by $\tau_{k\sim t}^{\mathbbm{i}}$ the time of the first reproduction event in colony $\mathbbm{i}$
after $t$ induced by the $k$th living particle at $t$ in
colony $\mathbbm{i}$.
Let $\rho_{k\sim t}$ be the time of first migration after $t$ from colony $\mathbbm{1}$ to colony $\mathbbm{2}$ caused by the motion of the $k$th living particle at $t$ in colony $\mathbbm{1}$. We denote that $\sum_{i = 1}^0 f_i = 0$ for any $f_i$.  The total number of migration and reproduction during $[t,t+h)$ is at most once a.s. if $h$ is small enough. Without abuse of notation, these events $\{\{\rho_{k\sim t}< t+h\},\{\tau_{k\sim t}^{\mathbbm{1}}< t+h\},\,k=1,2,\ldots\}$ are incompatible. It follows from the construction of the branching particle systems that
\beqnn
\left<\mu_{t+h}^{n,\mathbbm{1}},f\right>
\ar=\ar\frac{1}{n}\sum_{k\sim t}\(f\(X_{k\sim t}^{n,\mathbbm{1}}\(t+h\)\) +\Delta_{k\sim t}^{n,\mathbbm{1}}\(f\)-D_{k\sim t}^{n,\mathbbm{1}}\(f\)\)
\eeqnn
with
\begin{equation}\label{eq:delta1}
\Delta_{k\sim t}^{n, \mathbbm{1}}(f) =\Big[\sum_{i=1}^{\xi^{n, \mathbbm{1}}_{k\sim t}}f\(X_{k\sim t}^{n,\mathbbm{1},i}\(t+h\)\) - f\(X_{k\sim t}^{n,\mathbbm{1}}\(t+h\)\)\Big]\ind_{\{\tau_{k\sim t}^{\mathbbm{1}}< t+h\}}
\end{equation}
and
\beqnn
D_{k\sim t}^{n, \mathbbm{1}}(f) = f\(X_{k\sim t}^{n,\mathbbm{1}}\(t+h\)\)\ind_{\{\rho_{k\sim t}< t+h\}},
\eeqnn
where $\xi^{n, \mathbbm{1}}_{k\sim t}$, $k=1,2,\ldots$ are i.i.d. copies of $\xi^{n, \mathbbm{1}}$; $X_{k\sim t}^{n,\mathbbm{1},i}\(t+h\)$ has the same distribution
as $X_{k\sim t}^{n,\mathbbm{1}}\(t+h\)$; $\ind_{\{\cdot\}}$ is the indication function.
Applying It\^{o}'s formula, we have
\beqnn
f\(X_{k\sim t}^{n,\mathbbm{1}}\(t+h\)\)\ar=\ar f\(X_{k\sim t}^{n,\mathbbm{1}}\(t\)\)+\int_{t}^{t+h}f'\(X_{k\sim t}^{n,\mathbbm{1}}\(s\)\)dB^{\mathbbm{1}}_{k\sim t}\(s\)\ccr
\ar\ar +\frac12\int_{t}^{t+h}f''\(X_{k\sim t}^{n,\mathbbm{1}}\(s\)\)ds.
\eeqnn
Consequently,
\beqnn
\left<\mu_{t+h}^{n,\mathbbm{1}},f\right>
\ar=\ar\left<\mu_{t}^{n,\mathbbm{1}},f\right>
+\frac{1}{n}\sum_{k\sim t}\int_{t}^{t+h}f'\(X_{k\sim t}^{n,\mathbbm{1}}\(s\)\)dB^{\mathbbm{1}}_{k\sim t}\(s\)\cr
\ar\ar+\frac{1}{n}\sum_{k\sim t}\frac12\int_{t}^{t+h}f''\(X_{k\sim t}^{n,\mathbbm{1}}\(s\)\)ds
+\frac{1}{n}\sum_{k\sim t}\(\Delta_{k\sim t}^{n,\mathbbm{1}}\(f\)-D_{k\sim t}^{n,\mathbbm{1}}\(f\)\).
\eeqnn
Now we discretize the time interval according to step size $h$. Assume that the random events (including  branching, emigration and immigration events) only happen at the endpoints of each small subinterval.
Then, for any $jh\leq t<\(j+1\)h$ with $j=0,1,2,\ldots$, we have
\beqnn
\left<\mu_{t}^{n,\mathbbm{1}},f\right> \ar=\ar\left<\mu_{0}^{n,\mathbbm{1}},f\right>+\left<\mu_{t}^{n,\mathbbm{1}},f\right>-\left<\mu_{jh}^{n,\mathbbm{1}},f\right>+\sum_{\ell=0}^{j-1}\(\left<\mu_{\(\ell+1\)h}^{n,\mathbbm{1}},f\right>-\left<\mu_{\ell h}^{n,\mathbbm{1}},f\right>\)\cr
\ar=\ar\left<\mu_{0}^{n,\mathbbm{1}},f\right> +\sum_{\ell=0}^{j}\frac{1}{n}\sum_{k\sim \ell h}\int_{\ell h}^{ (\ell+1)h\wedge t}f'\(X_{k\sim \ell h}^{n,\mathbbm{1}}\(s\)\)dB_{k\sim \ell h}^{\mathbbm{1}}\(s\)\cr
\ar\ar+\sum_{\ell=0}^{j-1}\frac{1}{n}\sum_{k\sim \ell h}\(\Delta_{k\sim\ell h}^{n, \mathbbm{1}}(f)-D_{k\sim\ell h}^{n, \mathbbm{1}}(f)\)
+\sum_{\ell=0}^{j}\frac{1}{n}\sum_{k\sim \ell h}\frac12\int_{\ell h}^{(\ell+1)h\wedge t}f''\(X_{k\sim \ell h}^{n,\mathbbm{1}}\(s\)\)ds.
\eeqnn
Let $M_{1}^{n,f}, M_{2}^{n,f}$ and $M_{3}^{n,f}$ be martingales with
\beqnn
M_{1}^{n,f}\(t\)=\sum_{\ell=0}^{j}\frac{1}{n}\sum_{k\sim \ell h}\int_{\ell h}^{\(\ell+1\)h\wedge t}f'\(X_{k\sim \ell h}^{n,\mathbbm{1}}\(s\)\)dB^{\mathbbm{1}}_{k\sim \ell h}\(s\),
\eeqnn
\beqlb\label{M_2_1}
 M_{2}^{n,f}(t) \ar=\ar\sum_{\ell=0}^{j-1}\frac{1}{n}\sum_{k\sim \ell h}\Big\{\Delta_{k\sim\ell h}^{n, \mathbbm{1}}\(f\) - \mathbbm{E}\big[\Delta_{k\sim\ell h}^{n, \mathbbm{1}}\(f\)|\mathcal{F}_{\ell h}\big]\Big\}
\eeqlb
and
\beqlb\label{M_3_1}
 M_{3}^{n,f}(t) \ar=\ar\sum_{\ell=0}^{j-1}\frac{1}{n}\sum_{k\sim \ell h}\Big\{D_{k\sim\ell h}^{n, \mathbbm{1}}\(f\) - \mathbbm{E}\big[D_{k\sim\ell h}^{n, \mathbbm{1}}\(f\)|\mathcal{F}_{\ell h}\big]\Big\}.
\eeqlb
Moreover, $A^{n, f}$ is defined as
\beqlb\label{{A}}
A^{n,f}\(t\)\ar=\ar\sum_{\ell=0}^{\lfloor t/h\rfloor}\frac{1}{n}\sum_{k\sim \ell h}\frac12\int_{\ell h}^{\(\ell+1\)h\wedge t}f''\(X_{k\sim \ell h}^{n,\mathbbm{1}}\(s\)\)ds\cr
\ar\ar+\sum_{\ell=0}^{\lfloor t/h\rfloor-1}\frac{1}{n}\sum_{k\sim \ell h}\mathbbm{E}\Big[\Delta_{k\sim\ell h}^{n, \mathbbm{1}}\(f\)-D_{k\sim\ell h}^{n, \mathbbm{1}}\(f\)|\mathcal{F}_{\ell h}\Big]\cr
\ar=\ar\sum_{\ell=0}^{\lfloor t/h\rfloor}\frac12\int_{\ell h}^{\(\ell+1\)h\wedge t}\left<\mu_s^{n,\mathbbm{1}},f''\right>ds+\sum_{\ell=0}^{\lfloor t/h\rfloor-1}\frac{\be_{n,\mathbbm{1}}}{n}\lambda_{n,\mathbbm{1}}\left<\mu_{\ell h}^{n,\mathbbm{1}},f\right>h\cr
\ar\ar-\sum_{\ell=0}^{\lfloor t/h\rfloor-1}\left<\mu_{\ell h}^{n,\mathbbm{1}},\eta\(\cdot,\mu_{\ell h}^{n,\mathbbm{1}},\mu_{\ell h}^{n,\mathbbm{2}}\)f\right>h+O\(h\).
\eeqlb
One can see that
\beqlb\label{mu_1}
\langle \mu_t^{n,\mathbbm{1}}, f\rangle = \left<\mu_{0}^{n,\mathbbm{1}},f\right>+M_{1}^{n,f}(t)+M_{2}^{n,f}(t)-M_{3}^{n,f}(t)+A^{n,f}(t).
\eeqlb
Carrying out similar steps as above in colony $\mathbbm{2}$, we get
\begin{equation}\label{mu_2}
\begin{split}
&\left<\mu_{t}^{n,\mathbbm{2}},g\right>=\left<\mu_{0}^{n,\mathbbm{2}},g\right>
+\hat{M}_{1}^{n,g}\(t\)+\hat{M}_{2}^{n,g}\(t\)+\hat{M}_{3}^{n,g}\(t\)+\hat{A}^{n,g}\(t\),
\end{split}
\end{equation}
where
\beqnn
\hat{M}_{1}^{n,g}\(t\)=\sum_{\ell=0}^{\lfloor t/h\rfloor}\frac{1}{n}\sum_{k\sim \ell h}\int_{\ell h}^{\(\ell+1\)h\wedge t}g'\(X_{k\sim \ell h}^{n,\mathbbm{2}}\(s\)\)dB^{\mathbbm{2}}_{k\sim \ell h}\(s\)
\eeqnn
and
\beqlb\label{hat{M}_2_1}
\hat{M}_{2}^{n,g}\(t\) = \sum_{\ell=0}^{\lfloor t/h\rfloor-1}\frac{1}{n}\sum_{k\sim \ell h}\Big\{\Delta_{k\sim \ell h}^{n, \mathbbm{2}}\(g\)
-\mathbbm{E}\left[\Delta_{k\sim\ell h}^{n, \mathbbm{2}}\(g\)|\mathcal{F}_{\ell h}\right]\Big\}
\eeqlb
are martingales with $\Delta_{k\sim\ell h}^{n, \mathbbm{2}}\(g\)$ defined similar as (\ref{eq:delta1})
and $X_{k\sim t}^{n,\mathbbm{2},i}\(t+h\)$ having the same distribution
as $X_{k\sim t}^{n,\mathbbm{2}}\(t+h\)$. Moreover,
\begin{equation}\label{hat{M}_3_1}
\hat{M}_{3}^{n,g}\(t\)= \sum_{\ell=0}^{\lfloor t/h\rfloor-1}\frac{1}{n}\sum_{k\sim \ell h}\Big\{D_{k\sim \ell h}^{n, \mathbbm{2}}\(g\)
-\mathbbm{E}\left[D_{k\sim\ell h}^{n, \mathbbm{2}}\(g\)|\mathcal{F}_{\ell h}\right]\Big\}
\end{equation}
with
\beqnn
D_{k\sim\ell h}^{n, \mathbbm{2}}\(g\) = \langle\chi,g\rangle\ind_{\{ \rho_{k\sim\ell h}\leq\(\ell+1\)h\}}
\eeqnn
is also a martingale and
\beqlb\label{hat{A}}
\hat{A}^{n,g}\(t\)\ar=\ar\sum_{\ell=0}^{\lfloor t/h\rfloor}\frac{1}{n}\sum_{k\sim \ell h}\frac12\int_{\ell h}^{\(\ell+1\)h\wedge t}g''\(X_{k\sim \ell h}^{n,\mathbbm{2}}\(s\)\)ds\cr
 \ar\ar+ \sum_{\ell=0}^{\lfloor t/h\rfloor-1}\frac{1}{n}\sum_{k\sim \ell h}\mathbbm{E}\left[\Delta_{k\sim \ell h}^{n, \mathbbm{2}}\(g\)-D_{k\sim \ell h}^{n, \mathbbm{2}}\(g\)\Big|\mathcal{F}_{\ell h}\right]\cr
\ar=\ar\sum_{\ell=0}^{\lfloor t/h\rfloor}\frac12\int_{\ell h}^{\(\ell+1\)h\wedge t}\left<\mu_s^{n,\mathbbm{2}},g''\right>ds+\sum_{\ell=0}^{\lfloor t/h\rfloor-1}\frac{\be_{n,{\mathbbm{2}}}}{n}\lambda_{n,{\mathbbm{2}}}\left<\mu_{\ell h}^{n,\mathbbm{2}},g\right>h\cr
\ar\ar+\sum_{\ell=0}^{\lfloor t/h\rfloor-1}\langle\chi, g\rangle\langle \mu_{\ell h}^{n,\mathbbm{1}}, \eta(\cdot,\mu_{\ell h}^{n,\mathbbm{1}},\mu_{\ell h}^{n,\mathbbm{2}})\rangle h+O\(h\).
\eeqlb

From the construction of our model, one can check that $M_{1}^{n,f},M_{2}^{n,f}$ and $M_{3}^{n,f}$ (resp. $\hat{M}_{1}^{n,g},\hat{M}_{2}^{n,g}$ and $\hat{M}_{3}^{n,g}$) are three mutually uncorrelated martingales. The interaction may only exist between $M_{3}^{n,f}$ and $\hat{M}_{3}^{n,g}$. Subsequently, we calculate the quadratic variation and covariation processes for these martingales.
It follows from It\^{o}'s integral that
\begin{equation}\label{M_1}
\begin{split}
\left<M_{1}^{n,f}\right>_{t}&=\sum_{\ell=0}^{\lfloor t/h\rfloor}\frac{1}{n^2}\sum_{k\sim \ell h}\int_{\ell h}^{\(\ell+1\)h\wedge t}\left|f'\(X_{k\sim \ell h}^{n,\mathbbm{1}}\(s\)\)\right|^{2}ds\\
&=\frac{1}{n}\int_{0}^{t}\left<\mu_{s}^{n,\mathbbm{1}},\left|f'\right|^{2}\right>ds
{\rightarrow 0,}\qquad \text{~as~ $n\rightarrow\infty$}.
\end{split}
\end{equation}
Applying Lemma 8.12 in \cite{filter}, we obtain the quadratic variations of $M_{2}^{n,f}$ and $M_{3}^{n,f}$ as follows:
{\small\beqlb\label{M_2}
\left<M_{2}^{n,f}\right>_{t}
\ar=\ar\sum_{\ell=0}^{\lfloor t/h\rfloor-1}\frac{1}{n^2}\mathbbm{E}\left\{\Big[\sum_{k\sim \ell h}\big(\Delta_{k\sim\ell h}^{n, \mathbbm{1}}\(f\)-\mathbbm{E}(\Delta_{k\sim\ell h}^{n, \mathbbm{1}}\(f\)|\mathcal{F}_{\ell h})\big)\Big]^2\Big|\mathcal{F}_{\ell h}\right\}\cr
\ar=\ar\sum_{\ell=0}^{\lfloor t/h\rfloor-1}\frac{1}{n^2}\sum_{k\sim \ell h}\left[f^2\(X_{k\sim \ell h}^{n,\mathbbm{1}}\(\ell h\)\)\sigma^2_{n,\mathbbm{1}}\lambda_{n,\mathbbm{1}}h+O\(h^2\)\right]\cr
\ar=\ar\frac{1}{n}\left(\sum_{\ell=0}^{\lfloor t/h\rfloor-1}\left<\mu_{\ell h}^{n,\mathbbm{1}},f^2\right>\sigma^2_{n,{\mathbbm{1}}}\lambda_{n,{\mathbbm{1}}}h+O\(h\)\right)
\eeqlb}
and
\begin{equation}\label{M_3}
\begin{split}
\left<M_{3}^{n,f}\right>_{t}
=&\,\,\sum_{\ell=0}^{\lfloor t/h\rfloor-1}\frac{1}{n^2}\mathbbm{E}\left\{\Big[\sum_{k\sim \ell h}\big(D_{k\sim\ell h}^{n, \mathbbm{1}}\(f\)-\mathbbm{E}(D_{k\sim\ell h}^{n, \mathbbm{1}}\(f\)|\mathcal{F}_{\ell h})\big)\Big]^2\Big|\mathcal{F}_{\ell h}\right\}\\
=&\,\,\sum_{\ell=0}^{\lfloor t/h\rfloor-1}\frac{1}{n^2}\sum_{k\sim \ell h}\left[f^2(X_{k\sim \ell h}^{n,\mathbbm{1}}\(\ell h\))\eta\big(X_{k\sim \ell h}^{n,\mathbbm{1}}\(\ell h\),\mu_{\ell h}^{n,\mathbbm{1}},\mu_{\ell h}^{n,\mathbbm{2}}\big)h+O\(h^2\)\right]\\
=&\,\,\frac{1}{n}\Bigg(\sum_{\ell=0}^{\lfloor t/h\rfloor-1}\left<\mu_{\ell h}^{n,\mathbbm{1}},f^2\(\cdot\)\eta\(\cdot,\mu_{\ell h}^{n,\mathbbm{1}},\mu_{\ell h}^{n,\mathbbm{2}}\)\right>h+O\(h\)\Bigg).
\end{split}
\end{equation}
The quadratic variations of $\hat{M}_{1}^{n,g},\,\hat{M}_{2}^{n,g}$ and $\hat{M}_{3}^{n,g}$ can be similarly derived as follows:
\begin{equation}\label{hat{M}_1}
\begin{split}
\left<\hat{M}_{1}^{n,g}\right>_{t}&=\sum_{\ell=0}^{\lfloor t/h\rfloor}\frac{1}{n^2}\sum_{k\sim \ell h}\int_{\ell h}^{\(\ell+1\)h\wedge t}\Big|g'\Big(X_{k\sim \ell h}^{n,\mathbbm{2}}(s)\Big)\Big|^{2}ds\\
&=\frac{1}{n}\int_{0}^{t}\langle\mu_{s}^{n,\mathbbm{2}},|g'|^{2}\rangle ds{\rightarrow 0,}\qquad\text{~as~ $n\rightarrow\infty$},\\
\end{split}
\end{equation}
\begin{equation}\label{hat{M}_2}
\begin{split}
\left<\hat{M}_{2}^{n,g}\right>_{t}
=&\,\,\sum_{\ell=0}^{\lfloor t/h\rfloor-1}\frac{1}{n^2}\sum_{k\sim \ell h}\left(g^2\(X_{k\sim \ell h}^{n,\mathbbm{2}}\(\ell h\)\)\sigma^2_{n,{\mathbbm{2}}}\lambda_{n,\mathbbm{2}}h+O\(h^2\)\right)\\
=&\,\,\frac{1}{n}\left(\sum_{\ell=0}^{\lfloor t/h\rfloor-1}\left<\mu_{\ell h}^{n,\mathbbm{2}},g^2\right>\sigma^2_{n,{\mathbbm{2}}}\lambda_{n,\mathbbm{2}}h+O\(h\)\right)
\end{split}
\end{equation}
and
\begin{equation}\label{hat{M}_3}
\begin{split}
\left<\hat{M}_{3}^{n,g}\right>_{t}
=&\,\, \frac{1}{n}\sum_{\ell=0}^{\lfloor t/h\rfloor-1}\langle\chi, g\rangle^2\big\langle \mu_{\ell h}^{n,\mathbbm{1}}, \eta(\cdot,\mu_{\ell h}^{n,\mathbbm{1}},\mu_{\ell h}^{n,\mathbbm{2}})\big\rangle h+O\(h\).
\end{split}
\end{equation}

\begin{proposition}
The covariation process of ${M}_{3}^{n,f}$ and $\hat{M}_{3}^{n,g}$ is
\begin{equation*}
\begin{split}
\langle{M}_{3}^{n,f},\hat{M}_{3}^{n,g}\rangle_t
=&\sum_{\ell=0}^{\lfloor t/h\rfloor-1}\frac{1}{n}\langle \chi, g\rangle\big\langle\mu^{n,\mathbbm{1}}_{\ell h},f\(\cdot\)\eta(\cdot,\mu_{\ell h}^{n,\mathbbm{1}},\mu^{n,\mathbbm{2}}_{\ell h})\big\rangle h\\
&\quad-\sum_{\ell=0}^{\lfloor t/h\rfloor-1}\langle \chi, g\rangle\langle\mu^{n,\mathbbm{1}}_{\ell h},f\(\cdot\)\eta(\cdot,\mu^{n,\mathbbm{1}}_{\ell h},\mu^{n,\mathbbm{2}}_{\ell h})\rangle \big\langle\mu^{n,\mathbbm{1}}_{\ell h},\eta(\cdot,\mu^{n,\mathbbm{1}}_{\ell h},\mu^{n,\mathbbm{2}}_{\ell h})\big\rangle h^2
+ \frac{O\(h\)}{n}.
\end{split}
\end{equation*}
\end{proposition}

\begin{proof}
This result shows the correlation between the processes $(\mu_t^{n, \mathbbm{1}})_{t \ge 0}$ and $(\mu_t^{n, \mathbbm{2}})_{t \ge 0}.$
For simplicity of notation, we set
\beqnn
I_1\(\ell,k\)=\frac{1}{n}\Big\{D_{k\sim\ell h}^{n, \mathbbm{1}}\(f\) - \mathbbm{E}\big[D_{k\sim\ell h}^{n, \mathbbm{1}}\(f\)|\mathcal{F}_{\ell h}\big]\Big\}
\eeqnn
 and
 \beqnn
 I_2\(\ell,k\)=\,\frac{1}{n}\Big\{D_{k\sim\ell h}^{n, \mathbbm{2}}\(g\) - \mathbbm{E}\big[D_{k\sim\ell h}^{n, \mathbbm{2}}\(g\)|\mathcal{F}_{\ell h}\big]\Big\}.
 \eeqnn
It follows from  Equations (\ref{M_3_1}) and (\ref{hat{M}_3_1}) that
$${M}_{3}^{n,f}\(t\)\,=\,\sum_{\ell=0}^{\lfloor t/h\rfloor-1}I_1\(\ell\)\,=\,\sum_{\ell=0}^{\lfloor t/h\rfloor-1}\sum_{k\sim \ell h}I_1\(\ell,k\)$$
and
$$\hat{M}_{3}^{n,g}\(t\)\,=\,\sum_{\ell=0}^{\lfloor t/h\rfloor-1}I_2\(\ell\)\,=\,\sum_{\ell=0}^{\lfloor t/h\rfloor-1}\sum_{k\sim \ell h}I_2\(\ell,k\).$$
Since
$\mathbbm{E}I_1\(\ell\)=\mathbbm{E}I_2\(\ell\)=0$, we have
\beqnn
\<{M}_{3}^{n,f},\hat{M}_{3}^{n,g}\>_t
\ar=\ar\sum_{\ell=0}^{\lfloor t/h\rfloor-1}\mathbbm{E}\[I_1\(\ell\)I_2\(\ell\)|\mathcal{F}_{\ell h}\]\cr
\ar=\ar\sum_{\ell=0}^{\lfloor t/h\rfloor-1}\mathbbm{E}\Big[\sum_{k\sim\ell h}I_1\(\ell,k\)I_2\(\ell,k\)\Big|\mathcal{F}_{\ell h}\Big]\cr
\ar\ar +\sum_{\ell=0}^{\lfloor t/h\rfloor-1}\mathbbm{E}\Big[\sum_{\substack{k_1\sim \ell h, k_2\sim \ell h\\k_1\neq k_2}}I_1\(\ell,k_1\)I_2\(\ell,k_2\)\Big|\mathcal{F}_{\ell h}\Big].
\eeqnn
Note that
\beqnn
\ar\ar\mathbbm{E}\[I_1\(\ell,k\)I_2\(\ell,k\)\Big|\mathcal{F}_{\ell h}\]\cr
\ar\ar\qquad=\frac{\< \chi, g\>}{n^2} \left\{\mathbbm{E}\[D_{k\sim\ell h}^{n, \mathbbm{1}}\(f\)\big|\mathcal{F}_{\ell h}\]-
\mathbbm{E}\[D_{k\sim\ell h}^{n, \mathbbm{1}}\(f\)\big|\mathcal{F}_{\ell h}\]\cdot\mathbbm{E}\[{\text{\rm$\ind$}}_{\{ \rho_{k\sim\ell h}\leq(\ell+1)h\}}\big|\mathcal{F}_{\ell h}\]\right\}\cr
\ar\ar\qquad=\frac{\< \chi, g\>}{n^2} f\(X_{k\sim \ell h}^{n,\mathbbm{1}}(\ell h)\)\eta\(X_{k\sim \ell h}^{n,\mathbbm{1}}\(\ell h\),\mu^{n,\mathbbm{1}}_{\ell h},\mu^{n,\mathbbm{2}}_{\ell h}\)h\cr
\ar\ar\qquad\qquad\qquad\times \Big[1 - \eta\(X_{k\sim \ell h}^{n,\mathbbm{1}}(\ell h),\mu^{n,\mathbbm{1}}_{\ell h},\mu^{n,\mathbbm{2}}_{\ell h}\)h\Big]
+ \frac{O(h^2)}{n^2}.
\eeqnn
Similarly, for any $k_1\neq k_2$ one can see that
\beqnn
\ar\ar\mathbbm{E}\big[I_1\(\ell,k_1\)I_2\(\ell,k_2\)\big|\mathcal{F}_{\ell h}\big]\cr
\ar\ar\qquad=-\frac{\< \chi, g\>}{n^2} \mathbbm{E}\big[D_{{k_1}\sim\ell h}^{n, \mathbbm{1}}\(f\)\mid\mathcal{F}_{\ell h}\big]\cdot \mathbbm{E}\big[{\text{\rm$\ind$}}_{\{ \rho_{k_2\sim\ell h}\leq\(\ell+1\)h\}}\big|\mathcal{F}_{\ell h}\big]\cr
\ar\ar\qquad=-\frac{\< \chi, g\>}{n^2} f\(X_{k_1\sim \ell h}^{n,\mathbbm{1}}\(\ell h\)\)\eta\(X_{k_1\sim \ell h}^{n,\mathbbm{1}}\(\ell h\),\mu^{n,\mathbbm{1}}_{\ell h},\mu^{n,\mathbbm{2}}_{\ell h}\)\cr
\ar\ar\qquad\qquad\qquad\times\eta\(X_{k_2\sim \ell h}^{n,\mathbbm{1}}\(\ell h\),\mu^{n,\mathbbm{1}}_{\ell h},\mu^{n,\mathbbm{2}}_{\ell h}\)h^2
+\frac{O\(h^3\)}{n^2}.
\eeqnn
Therefore, we proceed to have
\beqnn
\ar\ar\<{M}_{3}^{n,f},\hat{M}_{3}^{n,g}\>_t\cr
\ar\ar\qquad=\sum_{\ell=0}^{\lfloor t/h\rfloor-1}\sum_{k\sim\ell h}\(\frac{1}{n^2}\langle \chi, g\rangle f\(X_{k\sim \ell h}^{n,\mathbbm{1}}\(\ell h\)\)\eta\(X_{k\sim \ell h}^{n,\mathbbm{1}}\(\ell h\),\mu^{n,\mathbbm{1}}_{\ell h},\mu^{n,\mathbbm{2}}_{\ell h}\)h+\frac{1}{n^2}O\(h^2\)\)\\
\ar\ar\qquad\qquad-\sum_{\ell=0}^{\lfloor t/h\rfloor-1}\sum_{\substack{k_1\sim \ell h\\k_2\sim \ell h}}\frac{1}{n^2}\<\chi, g\> f\(X_{k_1\sim \ell h}^{n,\mathbbm{1}}\(\ell h\)\)h^2\prod_{i = 1}^2\eta\(X_{k_i\sim \ell h}^{n,\mathbbm{1}}\(\ell h\),\mu^{n,\mathbbm{1}}_{\ell h},\mu^{n,\mathbbm{2}}_{\ell h}\)\\
\ar\ar\qquad=\sum_{\ell=0}^{\lfloor t/h\rfloor-1}\frac{1}{n}\<\chi, g\>\<\mu^{n,\mathbbm{1}}_{\ell h},f\(\cdot\)\eta\(\cdot,\mu^{n,\mathbbm{1}}_{\ell h},\mu^{n,\mathbbm{2}}_{\ell h}\)\> h\\
\ar\ar\qquad\qquad -\sum_{\ell=0}^{\lfloor t/h\rfloor-1}\< \chi, g\>\<\mu^{n,\mathbbm{1}}_{\ell h},f\(\cdot\)\eta\(\cdot,\mu^{n,\mathbbm{1}}_{\ell h},\mu^{n,\mathbbm{2}}_{\ell h}\)\> \<\mu^{n,\mathbbm{1}}_{\ell h},\eta\(\cdot,\mu^{n,\mathbbm{1}}_{\ell h},\mu^{n,\mathbbm{2}}_{\ell h}\)\> h^2 +\frac{O\(h\)}{n}.
\eeqnn
The result follows.
\end{proof}

In fact, the tightness of $\{\(\mu^{n,\mathbbm{1}},\mu^{n,\mathbbm{2}}\):n\geq 1\}$ in $D(\mathbbm{R}_{+},M_{F}\(\mathbbm{R}\)^2)$ is equivalent to the tightness of $\{\(\left<\mu^{n,\mathbbm{1}},f\right>,\left<\mu^{n,\mathbbm{2}},g\right>\):n\geq 1\}$
in $D(\mathbbm{R}_{+},\mathbbm{R}^2)$ for any $f,g\in {C_b^2\(\mathbbm{R}\)}$.
In the following we will make some estimations
(see Lemmas \ref{es1}--\ref{es3}) and then prove the tightness of the empirical measure for the branching particle systems.

\begin{lemma}\label{es1}
Assume that $\sup_{n}\mathbbm{E}[\langle\mu_0^{n, \mathbbm{1}}, 1\rangle^{2p} + \langle\mu_0^{n, \mathbbm{2}}, 1\rangle^{2p}] < \infty$ for any $p \ge 1.$ There exists a constant $K = K(p, T)$ such that
$$\sup_n\mathbbm{E}\Big[\sup_{t\leq T}\big(\langle\mu_t^{n,\mathbbm{1}},1\rangle^{2p} +
\langle\mu_t^{n,\mathbbm{2}},1\rangle^{2p}\big)\Big]\,<\,K.$$
\end{lemma}

\begin{proof}
Replacing $f$ and $g$ with $1$ in Equations (\ref{mu_1}) and \eqref{mu_2}, we have
\beqnn
\langle\mu_{t}^{n,\mathbbm{1}},1\rangle=\langle\mu_{0}^{n,\mathbbm{1}},1\rangle+M_{1}^{n,1}\(t\)+M_{2}^{n,1}\(t\)-M_{3}^{n,1}\(t\)+A^{n,1}\(t\),
\eeqnn
and
\beqnn
\langle\mu_{t}^{n,\mathbbm{2}},1\rangle=\langle\mu_{0}^{n,\mathbbm{2}},1\rangle+\hat{M}_{1}^{n,1}\(t\)+\hat{M}_{2}^{n,1}\(t\)+\hat{M}_{3}^{n,1}\(t\)+\hat{A}^{n,1}\(t\).
\eeqnn
Consequently, by (\ref{{A}}), (\ref{M_1})-(\ref{M_3}) we have
\beqnn
\mathbb{E}\langle\mu_{t}^{n,\mathbbm{1}}, 1\rangle^{2p} \ar\le\ar  K + K \mathbb{E}A^{n,1}(t)^{2p} + K\mathbb{E}\langle M_{1}^{n,1}\rangle_t^p+K\mathbb{E}\langle M_{2}^{n,1}\rangle_t^p + K\mathbb{E}\langle M_{3}^{n,1}\rangle_t^p\cr
\ar\le\ar  K + K \int_0^t \mathbb{E}\langle\mu_{s}^{n,\mathbbm{1}}, 1\rangle^{2p} d s
\eeqnn
for any $t \in [0, T].$ Moreover,  by \eqref{hat{A}}, \eqref{hat{M}_1}-\eqref{hat{M}_3} we have
\beqnn
\mathbb{E} \langle\mu_{t}^{n,\mathbbm{2}} ,1\rangle^{2p} &\le& K + K \mathbb{E}\hat{A}^{n,1}(t)^{2p} + K\mathbb{E}\langle \hat{M}_{1}^{n,1}\rangle_t^p+K\mathbb{E}\langle \hat{M}_{2}^{n,1}\rangle_t^p + K\mathbb{E}\langle \hat{M}_{3}^{n,1}\rangle_t^p\\
&\le& K + K\int_0^t \mathbb{E}\(\langle\mu_s^{n,\mathbbm{1}}, 1\rangle^{2p} +\langle\mu_s^{n,\mathbbm{2}},1\rangle^{2p}\)d s
\eeqnn
for any $t \in [0, T].$ Combining the above, by Gronwall's inequality we have
\beqlb\label{e}
\mathbb{E}\(\langle\mu_{t}^{n,\mathbbm{1}}, 1\rangle^{2p} + \langle \mu_{t}^{n,\mathbbm{2}} ,1\rangle^{2p} \)\le Ke^{Kt}.
\eeqlb
Moreover, it follows from Doob's inequality that
\beqnn
\mathbbm{E}\Big(\sup_{t\leq T}\langle\mu_{t}^{n,\mathbbm{1}},1\rangle^{2p}\Big)
&\leq &K \mathbbm{E}\langle\mu_{0}^{n,\mathbbm{1}},1\rangle^{2p}
+K\mathbbm{E}\big|M_{1}^{n,1}\(T\)\big|^{2p}+K\mathbbm{E}\big|M_{2}^{n,1}\(T\)\big|^{2p}\\
&&+K\mathbbm{E}\big|M_{3}^{n,1}\(T\)\big|^{2p}+K\mathbbm{E}\sup_{t\leq T}\big|A^{n,1}\(t\)\big|^{2p};\\
\mathbbm{E}\Big(\sup_{t\leq T}\langle\mu_{t}^{n,\mathbbm{2}},1\rangle^{2p}\Big)
&\leq &K \mathbbm{E}\langle\mu_{0}^{n,\mathbbm{2}},1\rangle^{2p}
+K\mathbbm{E}\big|\hat{M}_{1}^{n,1}\(T\)\big|^{2p}+K\mathbbm{E}\big|\hat{M}_{2}^{n,1}\(T\)\big|^{2p}\\
&&+K\mathbbm{E}\big|\hat{M}_{3}^{n,1}\(T\)\big|^{2p}+K\mathbbm{E}\sup_{t\leq T}\big|\hat{A}^{n,1}\(t\)\big|^{2p}.
\eeqnn
By Equations (\ref{{A}}), (\ref{M_1})-(\ref{M_3}), we have
\beqnn
\mathbbm{E}\Big(\sup_{t\leq T}\langle\mu_{t}^{n,\mathbbm{1}},1\rangle^{2p}\Big)
&\leq& K \mathbbm{E}\langle\mu_{0}^{n,\mathbbm{1}},1\rangle^{2p} + K \int_0^T \mathbbm{E}\langle\mu_s^{n,\mathbbm{1}},1\rangle^p\d s + K \int_0^T \mathbbm{E}\langle\mu_s^{n,\mathbbm{1}},1\rangle^{2p}\d s\\
&\leq& K_T + K \int_0^T \mathbbm{E}\langle\mu_s^{n,\mathbbm{1}},1\rangle^{2p}\d s,
\eeqnn
where the last inequality follows from \eqref{e} and $K_T$ is a positive constant with $T.$ Moreover, by \eqref{hat{A}}, \eqref{hat{M}_1}-\eqref{hat{M}_3} we have
\beqnn
\mathbbm{E}\Big(\sup_{t\leq T}\langle\mu_{t}^{n,\mathbbm{2}},1\rangle^{2p}\Big)
&\leq& K_T + K \int_0^T\mathbbm{E}\langle\mu_s^{n,\mathbbm{2}},1\rangle^{2p}\d s + K\int_0^T\mathbbm{E}\langle\mu_s^{n,\mathbbm{1}},1\rangle^{2p}\d s.
\eeqnn
By the above equations we have
\beqnn
\mathbbm{E}\sup_{t\leq T}\Big(\langle\mu_{t}^{n,\mathbbm{1}},1\rangle^{2p} + \langle\mu_{t}^{n,\mathbbm{2}},1\rangle^{2p}\Big) \le K_T + K\int_0^T \mathbb{E}\Big(\langle\mu_s^{n,\mathbbm{1}},1\rangle^{2p} + \langle\mu_s^{n,\mathbbm{2}},1\rangle^{2p}\Big)\d s.
\eeqnn
The result follows from Gronwall's inequality.
\end{proof}

\begin{lemma}\label{es2}
For any $t \ge s \ge 0$, $p \ge 1$ and $f,g\in {C_b^2\(\mathbbm{R}\)}$, we have
\beqnn
\mathbbm{E}\Big[\big|\langle{M_2}^{n,f}\rangle_t-\langle{M_2}^{n,f}\rangle_s\big|^p + \big|\langle{\hat{M}_2}^{n,g}\rangle_t-\langle{\hat{M}_2}^{n,g}\rangle_s\big|^p\Big]\leq K|t - s|^p.
\eeqnn
\end{lemma}

\begin{proof}
It follows from Equation (\ref{M_2}) 
and  H\"{o}lder inequality that
\begin{equation*}
\begin{split}
\mathbbm{E}\big|\langle{M_2}^{n,f}\rangle_t-\langle{M_2}^{n,f}\rangle_s\big|^p
&\leq \frac{1}{n^p}\mathbbm{E}\left(\sum_{l = \lfloor s/h\rfloor}^{\lfloor t/h\rfloor-1}\left<\mu_{\ell h}^{n,\mathbbm{1}},f^2\right>\sigma^2_{n,{\mathbbm{1}}}\lambda_{n,{\mathbbm{1}}}h+O\(h\)\right)^p\\
&\le K \mathbbm{E}\Big[\int_s^t \langle \mu_{r}^{n, \mathbbm{1}}, 1\rangle \d r\Big]^p \le K|t - s|^p.
\end{split}
\end{equation*}
The last inequalities comes from Lemma~\ref{es1}. Similarly estimations can be carried out
for $\langle{{\hat{M}}_2}^{n,g}\rangle_t$.
\end{proof}

\begin{lemma}\label{es3}
For any $t \ge s \ge 0$, $p \ge 1$ and $f,g\in {C_b^2\(\mathbbm{R}\)}$, we have
\beqnn
\mathbbm{E}|A^{n,f}\(t\)-A^{n,f}\(s\)|^{2p}\leq K|t - s|^{2p}
\eeqnn
and
\beqnn
\mathbbm{E}|\hat{A}^{n,g}\(t\)-\hat{A}^{n,g}\(s\)|^{2p}\leq K|t - s|^{2p}.
\eeqnn
\end{lemma}

\begin{proof}
It follows from Equation (\ref{{A}}) that
\beqnn
\ar\ar\mathbbm{E}|A^{n,f}\(t\)-A^{n,f}\(s\)|^{2p}\cr
 \ar\ar\qquad= \mathbbm{E}\left[\sum_{\ell=\lfloor s/h\rfloor}^{\lfloor t/h\rfloor}\frac12\int_{\ell h}^{\(\ell+1\)h\wedge t} \left<\mu_r^{n,\mathbbm{1}},f^{''}\right>dr +\sum_{\ell=\lfloor s/h\rfloor}^{\lfloor t/h\rfloor-1}\frac{\be_{n,\mathbbm{1}}}{n}\lambda_{n,\mathbbm{1}}\left<\mu_{\ell h}^{n,\mathbbm{1}},f\right>h\right.\cr
\ar\ar\qquad\qquad\qquad\left.-\sum_{\ell=\lfloor s/h\rfloor}^{\lfloor t/h\rfloor-1}\left<\mu_{\ell h}^{n,\mathbbm{1}},f\(\cdot\)\eta\(\cdot,\mu_{\ell h}^{n,\mathbbm{1}},\mu_{\ell h}^{n,\mathbbm{2}}\)\right>h+O\(h\)\right]^{2p}\cr
\ar\ar\qquad\le K\bigg[\int_s^t \langle \mu_{r}^{n,\mathbbm{1}}, 1\rangle d r\bigg]^{2p} \le K |t - s|^{2p}.
\eeqnn
The last inequality follows by Lemma \ref{es1}. Similar estimation can be carried out for $\hat{A}^{n,g}\(\cdot\)$.
\end{proof}

\begin{theorem}\label{th:existence}
Let $h:= h_n\rightarrow0$ as $n\rightarrow\infty$. The sequence $\left\{\(\mu^{n,\mathbbm{1}},\mu^{n,\mathbbm{2}}\):n\geq 1\right\}$
is tight in $D(\mathbbm{R}_{+},M_F(\mathbbm{R})^2)$.
Furthermore, the limit $\(\mu^{\mathbbm{1}},\mu^{\mathbbm{2}}\)$ lies in $C(\mathbbm{R}_{+},M_F(\mathbbm{R})^2)$ and is a solution
to MP (\ref{mp1},\,\ref{mp2}) with $b_{\mathbbm{1}}=\beta_{\mathbbm{1}}\lambda_{\mathbbm{1}},\ b_{\mathbbm{2}}=\beta_{\mathbbm{2}}\lambda_{\mathbbm{2}}, \gamma_{\mathbbm{1}}=\sigma^2_{\mathbbm{1}}\lambda_{\mathbbm{1}}$ and $\gamma_{\mathbbm{2}}=\sigma^2_{\mathbbm{2}}\lambda_{\mathbbm{2}}.$
\end{theorem}

\begin{proof}
For any $f, g \in {C_b^2(\mathbbm{R})}$ and $t \ge 0$ one can see that $\langle{M_3}^{n,f}\rangle_t \rightarrow 0$ and $\langle{\hat{M}_3}^{n,g}\rangle_t \rightarrow 0$ as $n \rightarrow \infty$ by \eqref{M_3} and \eqref{hat{M}_3}. For any $t \ge s \ge 0$ and $p > 1,$ by  H\"{o}lder inequality, Lemma~\ref{es2} and Lemma~\ref{es3} we have
\beqnn
\mathbb{E}|\langle\mu_t^{n,\mathbbm{1}}, f\rangle - \langle\mu_s^{n,\mathbbm{1}}, f\rangle|^{2p} &\le& K \mathbbm{E}|A^{n,f}\(t\)-A^{n,f}\(s\)|^{2p} + K  \mathbbm{E}\big|\langle{M_2}^{n,f}\rangle_t-\langle{M_2}^{n,f}\rangle_s\big|^p\\
&\le& K|t - s|^p + K|t - s|^{2p}
\eeqnn
and
\beqnn
\mathbb{E}|\langle\mu_t^{n,\mathbbm{2}}, g\rangle - \langle\mu_s^{n,\mathbbm{2}}, g\rangle|^{2p} &\le& K \mathbbm{E}|\hat{A}^{n,g}\(t\)-\hat{A}^{n,g}\(s\)|^{2p} + K \mathbbm{E}\big|\langle{\hat{M}_2}^{n,g}\rangle_t-\langle{\hat{M}_2}^{n,g}\rangle_s\big|^p\\
&\le& K|t - s|^p + K|t - s|^{2p}.
\eeqnn
The tightness of $\{\langle\mu_t^{n,\mathbbm{1}}, f\rangle: 0 \le t \le T\}, \{A^{n,f}(t): 0 \le t \le T\}, \{\langle M_2^{n, f}\rangle_t: 0 \le t \le T\}$ and $ \{\langle\mu_t^{n,\mathbbm{2}}, g\rangle: 0 \le t \le T\}, \{\hat{A}^{n,g}(t): 0 \le t \le T\}, \{\langle \hat{M}_2^{n, g}\rangle_t: 0 \le t \le T\}$ on $C([0, T], \mathbbm{R})$ follows from Theorem VI.4.1 in \cite{JS03}, which implies that $\{(\mu_t^{n, \mathbbm{1}}, \mu_t^{n, \mathbbm{2}}): 0 \leq t \leq T\}$ is tight on $C([0, T], M_{F}(\mathbbm{R})^2).$ Hence there is a subsequence $\{(\mu_t^{n_k, \mathbbm{1}}, \mu_t^{n_k, \mathbbm{2}}): 0 \leq t \leq T\}$ converges in law. Suppose that $(\mu^{\mathbbm{1}}, \mu^{\mathbbm{2}})$ is the weak limit of $(\mu^{n_k, \mathbbm{1}}, \mu^{n_k, \mathbbm{2}})$ as $k\rightarrow\infty$. For any $f, g \in C_b^2(\mathbbm{R})$, we have
\beqnn
&&\Big(\langle\mu^{n_k, \mathbbm{1}}, f\rangle, A^{n_k, f}, \langle M_2^{n_k, f}\rangle, \langle\mu^{n_k, \mathbbm{2}}, g\rangle,  \hat{A}^{n_k, g}, \langle \hat{M}_2^{n_k, g}\rangle\Big)\\
&&\qquad\qquad \longrightarrow\Big(\langle\mu^{\mathbbm{1}}, f\rangle, A^{f}, \langle M_2^{f}\rangle, \langle\mu^{\mathbbm{2}}, g\rangle,  \hat{A}^{g}, \langle \hat{M}_2^{g}\rangle\Big)
\eeqnn
weakly as $k \rightarrow \infty.$ Furthermore, $\langle M_2^{n_k,f}\rangle_t \rightarrow \langle M^f_2\rangle_t$, $\langle\hat{M}_2^{n_k,g}\rangle_t \rightarrow \langle\hat{M}_2^g\rangle_t$ as $k \rightarrow \infty$, and for $t \le T,$
\beqnn
\langle M_2^{n_k,f}\rangle_t \le K\big(1 + \sup_{0 \le t\le T}\<\mu_t^{\mathbbm{1}}, 1\>\big),\quad  \langle\hat{M}_2^{n_k,g}\rangle_t \le K\big(1 + \sup_{0 \le t\le T}\<\mu_t^{\mathbbm{2}}, 1\>\big).
\eeqnn
By Lemma~\ref{es1} we can pass to the limit to conclude that $M^f_2(t)$ and $\hat{M}_2^g(t)$ are martingales. Let $b_{\mathbbm{1}}=\beta_{\mathbbm{1}}\lambda_{\mathbbm{1}}$, $b_{\mathbbm{2}}=\beta_{\mathbbm{2}}\lambda_{\mathbbm{2}}$, $\gamma_{\mathbbm{1}}=\sigma^2_{\mathbbm{1}}\lambda_{\mathbbm{1}}$
and $\gamma_{\mathbbm{2}}=\sigma^2_{\mathbbm{2}}\lambda_{\mathbbm{2}}$. One can see that
\beqnn
\langle M^{f}_2\rangle_t=\gamma_{\mathbbm{1}}\int_0^t\langle\mu_s^{\mathbbm{1}},f^2\rangle ds
\qquad\text{
and}\qquad
\langle{\hat{M}}^{g}_2\rangle_t=\gamma_{\mathbbm{2}}\int_0^t\langle \mu_s^{\mathbbm{2}},g^2\rangle ds
\eeqnn
by \eqref{M_2} and \eqref{hat{M}_2}. It implies that  $\(\mu^{\mathbbm{1}},\,\mu^{\mathbbm{2}}\)$  is a solution to MP (\ref{mp1},\,\ref{mp2}). The result follows.
\end{proof}

\section{Uniqueness of solution to the martingale problem}\label{uniqueness}
In this section, we first derive the SPDEs satisfied by the distribution-function-valued processes of the mutually interacting superprocesses with migration, and then establish its equivalence with MP (\ref{mp1},\,\ref{mp2}). Moreover, the pathwise uniqueness of the SPDEs is proved by an extended Yamada-Watanabe argument.

\subsection{A related system of SPDEs}

 For any $y\in\mathbbm{R}$, we write
\begin{equation}\label{eq:rfu}
u_t^{\mathbbm{1}}\(y\)=\mu_t^{\mathbbm{1}}\((-\infty,y]\)\,\qquad\text{and}\qquad\,u_t^{\mathbbm{2}}\(y\)=\mu_t^{\mathbbm{2}}\((-\infty,y]\)
\end{equation}
as the distribution-function-valued processes for the mutually interacting superprocesses with migration $\(\mu_t^{\mathbbm{1}},\mu_t^{\mathbbm{2}}\)_{t\geq0}$. For any $x,y\in\mathbbm{R}\cup\{\pm\infty\}$, $\nu_1,\nu_2\in M_{F}\(\mathbbm{R}\)$ and $\eta\(\cdot,\nu_1,\nu_2\)\in C_b^{+}\(\mathbbm{R}\times M_F(\mathbbm{R})^2\)$, denote by
$\xi\(y,\nu_1,\nu_2\)=\int_{-\infty}^{y}\eta\(x,\nu_1,\nu_2\)\nu_1\(dx\).$
Let $W^{\mathbbm{i}}(\d s\,\d a)$
be independent space-time white noise random measures on $\mathbbm{R}_{+}\times\mathbbm{R}$ with intensity $ds\,da$ and $\mathbbm{i}\in\{\mathbbm{1,2}\}$. We consider the following SPDEs: for any
$t\in\mathbbm{R}_{+}$ and $y\in\mathbbm{R}$,
\begin{equation}\label{eq:spdes}
\begin{cases}
u_t^{\mathbbm{1}}\(y\)&=u_0^{\mathbbm{1}}\(y\)+\sqrt{\gamma_{\mathbbm{1}}}\int_0^t\int_{0}^{u_s^{\mathbbm{1}}\(y\)}W^{\mathbbm{1}}\(d s\, d a\)
+\int_0^t\big(\frac{\Delta}{2} u_s^{\mathbbm{1}}\(y\)+b_{\mathbbm{1}}u_s^{\mathbbm{1}}\(y\)\big)d s \\
&\quad-\int_0^t \xi\(y,\mu_s^{\mathbbm{1}},\mu_s^{\mathbbm{2}}\)ds,\\
u_t^{\mathbbm{2}}\(y\)&=u_0^{\mathbbm{2}}\(y\)+\sqrt{\gamma_{\mathbbm{2}}}\int_0^t\int_{0}^{u_s^{\mathbbm{2}}\(y\)}W^{\mathbbm{2}}\(ds\, da\)
+\int_0^t\big(\frac{\Delta}{2} u_s^{\mathbbm{2}}\(y\)+b_{\mathbbm{2}}u_s^{\mathbbm{2}}\(y\)\big) d s\\
&\quad +\dot{\chi}\(y\)\xi\(+\infty,\mu_s^{\mathbbm{1}},\mu_s^{\mathbbm{2}}\)ds,
\end{cases}
\end{equation}
where $\dot{\chi}$ is the distribution function of $\chi$, i.e., $\dot{\chi}\(y\)=\chi\((-\infty,y]\)$.
\begin{definition}
The SPDEs (\ref{eq:spdes}) have a weak solution if there exists a $C_{b,m}\(\mathbbm{R}\)^2$-valued process $\(u_t^{\mathbbm{1}},u_t^{\mathbbm{2}}\)_{t\geq0}$ on a stochastic basis such that for any $f,\,g\in C_0^2\(\mathbbm{R}\)$ and $t\ge 0,$ the following holds:
\begin{equation}\label{eq:fgspdes}
\begin{cases}
\langle u_t^{\mathbbm{1}},f\rangle_1=\langle u_0^{\mathbbm{1}},f\rangle_1 +\sqrt{\gamma_{\mathbbm{1}}}\int_0^t\int_{0}^{\infty}\int_{\mathbbm{R}}f\(y\){\text{\rm$\ind$}}_{\left\{a\leq u_s^{\mathbbm{1}}\(y\)\right\}}dyW^{\mathbbm{1}}\(ds\, da\)\\
\,\,\,\,\,\,\,\,\,\,\,\,\,\,\,\,\,\,\,\,\,\,\,\,\,\,\,\,\,\,\,\,
+\int_0^t\big[ \langle \frac{\triangle}{2} u_s^{\mathbbm{1}},f\rangle_1 +b_{\mathbbm{1}}\langle u_s^{\mathbbm{1}},f\rangle_1 - \langle\xi\(\cdot,\mu_{s}^{\mathbbm{1}},\mu_{s}^{\mathbbm{2}}\),f\rangle_1\big]ds,\\
\langle u_t^{\mathbbm{2}},g\rangle_1=
\langle u_0^{\mathbbm{2}},g\rangle_1 +\sqrt{\gamma_{\mathbbm{2}}}\int_0^t\int_{0}^{\infty}\int_{\mathbbm{R}}g\(y\)
{\text{\rm$\ind$}}_{\left\{a\leq u_s^{\mathbbm{2}}\(y\)\right\}}dyW^{\mathbbm{2}}\(ds\, da\)\\
\,\,\,\,\,\,\,\,\,\,\,\,\,\,\,\,\,\,\,\,\,\,\,\,\,\,\,\,\,\,\,\,
+\int_0^t\big[ \langle  \frac{\triangle}{2}u_s^{\mathbbm{2}}, g\rangle_1 +b_{\mathbbm{2}}\langle u_s^{\mathbbm{2}},g\rangle_1+\langle\dot\chi,g\rangle_1\xi\(+\infty,\mu_{s}^{\mathbbm{1}},\mu_{s}^{\mathbbm{2}}\)\big]ds.
\end{cases}
\end{equation}
\end{definition}

\begin{proposition}\label{prop11}
Suppose that $\(u_t^{\mathbbm{1}},u_t^{\mathbbm{2}}\)_{t\geq0}$ is a solution to the system of SPDEs (\ref{eq:spdes}). Then the corresponding measure-valued process $\(\mu_t^{\mathbbm{1}},\mu_t^{\mathbbm{2}}\)_{t\geq 0}$ is a solution to MP (\ref{mp1}, \ref{mp2}).
\end{proposition}
\proof
For a non-decreasing continuous function $h$ on $\mathbbm{R}$, the inverse function is defined as
$h^{-1}\(a\)=\inf\{x:h\(x\)>a\}.$ Then for any $f,g\in C_0^3\(\mathbbm{R}\)$, we have
\begin{equation*}
\begin{split}
\langle\mu_t^{\mathbbm{1}},f\rangle=&-\langle u_t^{\mathbbm{1}},f'\rangle_1\\
=&-\langle u_0^{\mathbbm{1}},f'\rangle_1 -\sqrt{\gamma_{\mathbbm{1}}}\int_0^t\int_{0}^{\infty}\int_{\mathbbm{R}}f'\(y\){\text{\rm$\ind$}}_{\left\{a\leq u_s^{\mathbbm{1}}\(y\)\right\}}dyW^{\mathbbm{1}}\(ds\, da\)\\
&\qquad-\int_0^t\( \langle \frac{\triangle}{2} u_s^{\mathbbm{1}},f'\rangle_1 +b_{\mathbbm{1}}\langle u_s^{\mathbbm{1}},f'\rangle_1-\langle\xi\(\cdot,\mu_s^{\mathbbm{1}},\mu_s^{\mathbbm{2}}\),f'\rangle_1\)ds\\
=&\langle\mu_0^{\mathbbm{1}},f\rangle+\sqrt{\gamma_{\mathbbm{1}}}\int_0^t\int_{0}^{\infty}f\(\(u_s^{\mathbbm{1}}\)^{-1}\(a\)\)W^{\mathbbm{1}}\(ds\, da\)\\
&\qquad+\int_0^t\( \langle \mu_s^{\mathbbm{1}},\frac12f''\rangle+b_{\mathbbm{1}}\langle \mu_s^{\mathbbm{1}},f\rangle -\langle\mu_s^{\mathbbm{1}},\eta\(\cdot,\mu_s^{\mathbbm{1}},\mu_s^{\mathbbm{2}}\)f\rangle\)ds
\end{split}
\end{equation*}
and
\begin{equation*}
\begin{split}
\langle\mu_t^{\mathbbm{2}},g\rangle=&-\langle u_t^{\mathbbm{2}},g'\rangle_1\\
=&-\langle u_0^{\mathbbm{2}},g'\rangle_1 -\sqrt{\gamma_{\mathbbm{2}}}\int_0^t\int_{0}^{\infty}\int_{\mathbbm{R}}g'\(y\){\text{\rm$\ind$}}_{\left\{a\leq u_s^{\mathbbm{2}}\(y\)\right\}}dyW^{\mathbbm{2}}\(ds\, da\)\\
&\qquad-\int_0^t\( \langle \frac{\triangle}{2} u_s^{\mathbbm{2}},g'\rangle_1 +b_{\mathbbm{2}}\langle u_s^{\mathbbm{2}},g'\rangle_1 + \langle\dot\chi,g' \rangle_1\xi\(+\infty,\mu_s^{\mathbbm{1}},\mu_s^{\mathbbm{2}}\)\)ds\\
=&\langle \mu_0^{\mathbbm{2}},g\rangle+\sqrt{\gamma_{\mathbbm{2}}}\int_0^t\int_{0}^{\infty}g\(\(u_s^{\mathbbm{2}}\)^{-1}\(a\)\)W^{\mathbbm{2}}\(ds\, da\)\\
&\qquad+\int_0^t\( \langle \mu_s^{\mathbbm{2}},\frac12g''\rangle+b_{\mathbbm{2}}\langle \mu_s^{\mathbbm{2}},g\rangle +\langle\chi,g\rangle\xi\(+\infty,\mu_s^{\mathbbm{1}},\mu_s^{\mathbbm{2}}\)\)ds.
\end{split}
\end{equation*}
Thus, $M_t^f$ and $\hat{M}_t^g$ are martingales with quadratic variation processes
\begin{equation*}
\begin{split}
\langle M^f\rangle_t=&\,\,\gamma_{\mathbbm{1}}\int_0^t\int_0^{\infty}f^2\(\(u_s^{\mathbbm{1}}\)^{-1}\(a\)\)ds\,da\\
=&\,\,\gamma_{\mathbbm{1}}\int_0^t\int_{\mathbbm{R}}f^2\(y\)dsd\(u_s^{\mathbbm{1}}\(y\)\)
=\gamma_{\mathbbm{1}}\int_0^t\langle\mu_s^{\mathbbm{1}},f^2\rangle ds
\end{split}
\end{equation*}
and
\begin{equation*}
\begin{split}
\langle \hat{M}^g\rangle_t=&\,\,\gamma_{\mathbbm{2}}\int_0^t\int_0^{\infty}g^2\(\(u_s^{\mathbbm{2}}\)^{-1}\(a\)\)ds\,da\\
=&\,\,\gamma_{\mathbbm{2}}\int_0^t\int_{\mathbbm{R}}g^2\(y\)dsd\(u_s^{\mathbbm{2}}\(y\)\)
=\gamma_{\mathbbm{2}}\int_0^t\langle\mu_s^{\mathbbm{2}},g^2\rangle ds.
\end{split}
\end{equation*}
The independence of $W^{\mathbbm{1}}$ and $W^{\mathbbm{2}}$ leads to $\langle M^f, \hat{M}^g\rangle_t=0.$ Therefore, $\(\mu_t^{\mathbbm{1}},\mu_t^{\mathbbm{2}}\)_{t\geq 0}$ is a solution to MP (\ref{mp1},\,\ref{mp2}). That completes the proof.
\qed

\begin{proposition}\label{prop2}
Suppose that $\(\mu_t^{\mathbbm{1}},\mu_t^{\mathbbm{2}}\)_{t\geq 0}$ is a solution to MP (\ref{mp1}, \ref{mp2}) and ${\eta}\(\cdot,\nu_1,\nu_2\)\in C_0^{1}\(\mathbbm{R}\)$ for any $\nu_1,\nu_2\in M_F\(\mathbbm{R}\)$. Then
the random field $\(u_t^{\mathbbm{1}},u_t^{\mathbbm{2}}\)_{t\geq0}$ defined by (\ref{eq:rfu}) is a weak solution to SPDEs (\ref{eq:spdes}).
\end{proposition}

\proof
Let $f, g\in C_0^2\(\mathbbm{R}\)$ and set $\tilde{f}\(y\)=\int_y^{\infty}f\(x\)dx$, $\tilde{g}\(y\)=\int_y^{\infty}g\(x\)dx$. Then we have
\begin{equation*}
\begin{split}
\langle u_t^{\mathbbm{1}}, f\rangle_1
=&\,\,\langle\mu_t^{\mathbbm{1}},\tilde{f}\rangle\\
=&\,\,\langle\mu_0^{\mathbbm{1}},\tilde{f}\rangle+\int_0^t\langle\mu_s^{\mathbbm{1}},\frac12{\tilde{f}}^{''}\rangle ds
+b_{\mathbbm{1}}\int_0^t\langle\mu_s^{\mathbbm{1}},\tilde{f}\rangle ds-\int_0^t\langle\mu_s^{\mathbbm{1}},{\eta}\(\cdot,\mu_s^{\mathbbm{1}},\mu_s^{\mathbbm{2}}\)
\tilde{f}\rangle ds+M_t^{\tilde{f}}\\
=&\,\, \langle u_0^{\mathbbm{1}},{f}\rangle_1+\int_0^t\langle u_s^{\mathbbm{1}},\frac12{{f}}'' \rangle_1 ds
+b_{\mathbbm{1}}\int_0^t\langle u_s^{\mathbbm{1}},{f}\rangle_1 ds
+\int_0^t\big\langle u_s^{\mathbbm{1}},({\eta}(\cdot,\mu_s^{\mathbbm{1}},\mu_s^{\mathbbm{2}})
{\tilde{f}})'\big\rangle_1 ds+M_t^{\tilde{f}}.
\end{split}
\end{equation*}
Note that
\begin{equation*}
\begin{split}
\langle u_s^{\mathbbm{1}},({\eta}(\cdot,\mu_s^{\mathbbm{1}},\mu_s^{\mathbbm{2}})
{\tilde{f}})'\rangle_1
=&\,\,\langle u_s^{\mathbbm{1}},{\eta}'\(\cdot,\mu_s^{\mathbbm{1}},\mu_s^{\mathbbm{2}}\)\tilde{f}-{\eta}\(\cdot,\mu_s^{\mathbbm{1}},\mu_s^{\mathbbm{2}}\)f\rangle_1\\
=&\,\,\langle u_s^{\mathbbm{1}}{\eta}'\(\cdot,\mu_s^{\mathbbm{1}},\mu_s^{\mathbbm{2}}\),\tilde{f}\rangle_1 - \langle u_s^{\mathbbm{1}},{\eta}\(\cdot,\mu_s^{\mathbbm{1}},\mu_s^{\mathbbm{2}}\)f\rangle_1\\
=&\,\,\Big\langle \int_{-\infty}^{\cdot}u_s^{\mathbbm{1}}\(x\){\eta}'\(x,\mu_s^{\mathbbm{1}},\mu_s^{\mathbbm{2}}\)dx,{f}\Big\rangle_1-\langle u_s^{\mathbbm{1}},{\eta}\(\cdot,\mu_s^{\mathbbm{1}},\mu_s^{\mathbbm{2}}\)f\rangle_1\\
=&\,\,-\Big\langle\int_{-\infty}^{\cdot}{\eta}\(x,\mu_s^{\mathbbm{1}},\mu_s^{\mathbbm{2}}\)du_s^{\mathbbm{1}}\(x\),f\Big\rangle_1\\
=&\,\,-\langle\xi\(\cdot,\mu_s^{\mathbbm{1}},\mu_s^{\mathbbm{2}}\),f\rangle_1.
\end{split}
\end{equation*}
Therefore, we continue to have
\beqlb\label{eq:u_t^1}
\langle u_t^{\mathbbm{1}}, f\rangle_1\ar=\ar\langle u_0^{\mathbbm{1}},{f}\rangle_1+\int_0^t\langle u_s^{\mathbbm{1}},\frac12{{f}}''\rangle_1 ds\cr
\ar\ar\quad +\  b_{\mathbbm{1}}\int_0^t\langle u_s^{\mathbbm{1}},{f}\rangle_1 ds-\int_0^t\langle\xi\(\cdot,\mu_s^{\mathbbm{2}},\mu_s^{\mathbbm{1}}\),f\rangle_1 ds+M_t^{\tilde{f}}.
\eeqlb
Similarly, one shall have
\beqlb\label{eq:u_t^2}
\langle u_t^{\mathbbm{2}}, g\rangle_1
\ar=\ar \langle u_0^{\mathbbm{2}},{g}\rangle_1 + \int_0^t\langle u_s^{\mathbbm{2}},\frac12{{g}}^{''}\rangle_1 ds
+b_{\mathbbm{2}}\int_0^t\langle u_s^{\mathbbm{2}},{g}\rangle_1 ds\cr
\ar\ar\quad +\langle\dot\chi,{g}\rangle_1\int_0^t \xi\(+\infty,\mu_s^{\mathbbm{1}},\mu_s^{\mathbbm{2}}\)
ds+\hat{M}_t^{\tilde{g}}.
\eeqlb
Let $S'\(\mathbbm{R}\)$ be the space of Schwarz distribution and define the $S'\(\mathbbm{R}\)$-valued processes $N_t$ and $\hat{N}_t$ by $N_t\(f\)=M_t^{\tilde{f}}$ and $\hat{N}_t\(g\)=\hat{M}_t^{\tilde{g}}$ for any $f,g\in C_0^{\infty}\(\mathbbm{R}\)$. Then $N_t$ and $\hat{N}_t$ are $S'\(\mathbbm{R}\)$-valued continuous square-integrable martingales with
\begin{equation*}
\begin{split}
\langle N\(f\)\rangle_t=&\,\,\langle M^{\tilde{f}}\rangle_t=
\gamma_{\mathbbm{1}}\int_0^t\int_{\mathbbm{R}}\tilde{f}^2\(y\)\mu_s^{\mathbbm{1}}\(dy\)ds\\
=&\int_0^t\int_{0}^{\infty}\(\sqrt{\gamma_{\mathbbm{1}}}\)^2\tilde{f}^2\(\(u_s^{\mathbbm{1}}\)^{-1}\(a\)\)da\,ds\\
=&\int_0^t\int_{0}^{\infty}\(\sqrt{\gamma_{\mathbbm{1}}}\int_{\mathbbm{R}}\ind_{\{a\leq u_s^{\mathbbm{1}}\(y\)\}}f\(y\)dy\)^2da\,ds
\end{split}
\end{equation*}
and
\begin{equation*}
\begin{split}
\langle \hat{N}\(g\)\rangle_t=&\,\,\langle \hat{M}^{\tilde{g}}\rangle_t=\gamma_{\mathbbm{2}}\int_0^t\int_{\mathbbm{R}}\tilde{g}^2\(y\)\mu_s^{\mathbbm{2}}\(dy\)ds\\
=&\,\,\int_0^t\int_{0}^{\infty}\(\sqrt{\gamma_{\mathbbm{2}}}\)^2\tilde{g}^2\(\(u_s^{\mathbbm{2}}\)^{-1}\(a\)\)da\,ds\\
=&\,\,\int_0^t\int_{0}^{\infty}\(\sqrt{\gamma_{\mathbbm{2}}}\int_{\mathbbm{R}}\ind_{\{a\leq u_s^{\mathbbm{2}}\(y\)\}}g\(y\)dy\)^2da\,ds.
\end{split}
\end{equation*}
Moreover, one can see that
$\langle{N}\(f\),\hat{N}\(g\)\rangle_t=\langle{M}^{\tilde{f}},\hat{M}^{\tilde{g}}\rangle_t=0.$
By Theorem~III-7 and Corollary~III-8 in \cite{EM90}, on some extension of the probability space, one can define
two independent Gaussian white noises $W^{\mathbbm{i}}(d s\, d a)$, $\mathbbm{i}=\mathbbm{1,2}$ on $\mathbbm{R}_{+}\times\mathbbm{R}$ with intensity $ds\,da$ such that
 $$N_t\(f\)=\int_0^t\int_0^{\infty}\int_{\mathbbm{R}}\sqrt{\gamma_{\mathbbm{1}}}\ind_{\{a\leq u_s^{\mathbbm{1}}\(x\)\}}f\(x\)dxW^{\mathbbm{1}}\(ds\,da\)$$
and
$$\hat{N}_t\(g\)=\int_0^t\int_0^{\infty}\int_{\mathbbm{R}}\sqrt{\gamma_{\mathbbm{2}}}\ind_{\{a\leq u_s^{\mathbbm{2}}\(x\)\}}g\(x\)dxW^{\mathbbm{2}}\(ds\,da\).$$
Plugging back to (\ref{eq:u_t^1}) and (\ref{eq:u_t^2}), one can see that $\(u^{\mathbbm{1}}_t,u^{\mathbbm{2}}_t\)_{t\geq0}$ is a solution to (\ref{eq:spdes}).
\qed

\subsection{Uniqueness for SPDEs}
This subsection is devoted to prove the pathwise uniqueness for the solution to the system of SPDEs (\ref{eq:spdes}). By Propositions \ref{prop11} and \ref{prop2}, the uniqueness for the solution to MP (\ref{mp1}, \ref{mp2}) is then a direct consequence. We apply the approach of an extended Yamada-Watanabe argument to smooth function. This is an adaptation to that of Proposition 3.1 in \cite{MX2015}.

Before we go deep into the uniqueness theorem, let's introduce some notation.  Let $\Phi\in C_0^{\infty}\(\mathbbm{R}\)^{+}$  such that supp$\(\Phi\)\subset\(-1,1\)$
and the total integral is $1$. Define $\Phi_m\(x\)=m\Phi\(mx\)$. Notice that $\lim_{m \rightarrow \infty}\Phi_m(x)= \delta_0(x),$ which is the Dirac-$\delta$ function on $\mathbbm{R}.$
Let $\{a_k\}$ be a decreasing sequence defined recursively by $a_0=1$ and $\int_{a_k}^{a_{k-1}}z^{-1}dz=k$ for $k\geq 1.$ Let $\psi_k$ be non-negative functions in $C_0^{\infty}\(\mathbbm{R}\)$ such that supp$\(\psi_k\)\subset\(a_k,a_{k-1}\)$ and $\int_{a_k}^{a_{k-1}}\psi_k\(z\)dz=1$ and $\psi_k\(z\)\leq 2\(kz\)^{-1}$ for all $z\in\mathbbm{R}.$
Denote by
$$\phi_k\(z\)=\int_0^{|z|}dy\int_0^{y}\psi_k\(x\)dx,\,\,\,\,\forall\ z\in\mathbbm{R}.$$
Then, $\phi_k\(z\)\uparrow |z|$, $|\phi_k'\(z\)|\leq 1$ and $|z|\phi_k''\(z\)\leq 2k^{-1}.$
Let
$$J\(x\)=\int_{\mathbbm{R}}e^{-\left|x\right|}\varrho\(x-y\)dy,$$
where $\varrho$ is the mollifier given by
$\varrho\(x\)=C\exp\{-1/\(1-x^2\)\}\ind_{\{\left|x\right|<1\}},$
and $C$ is a constant such that $\int_{\mathbbm{R}}\varrho\(x\)dx=1$. Then, for any $m\in{\mathbbm{Z}_{+}}$, there
are positive constants $c_m$ and $C_m$ such that
\begin{equation}\label{eq:J}
c_me^{-\left|x\right|}\leq \left|J^{\(m\)}\(x\)\right|\leq C_me^{-\left|x\right|},\,\,\,\,\forall\ x\in\mathbbm{R}.
\end{equation}

Assume that $\(u_t^{\mathbbm{1}},u_t^{\mathbbm{2}}\)_{t\geq 0}$ and $\left({\tilde{u}}_t^{\mathbbm{1}},{\tilde{u}}_t^{\mathbbm{2}}\right)_{t\geq 0}$
are two solutions to the system of SPDEs (\ref{eq:spdes}) with the same initial values. $(\mu_t^{\mathbbm{1}}, \mu_t^{\mathbbm{2}})_{t \ge 0}$ and $({\tilde{\mu}}_t^{\mathbbm{1}},{\tilde{\mu}}_t^{\mathbbm{2}})_{t \ge 0}$ stand for their corresponding measure-valued processes,
namely, $u^{\mathbbm{i}}_t\(y\)=\mu^{\mathbbm{i}}_t(-\infty, y]$ and $\tilde{u}^{\mathbbm{i}}_t\(y\)=\tilde{\mu}^{\mathbbm{i}}_t(-\infty, y]$
for $\mathbbm{i}=\mathbbm{1,2}$. Let $v_t^{\mathbbm{i}}(y) =u_t^{\mathbbm{i}}(y)-{\tilde{u}}_t^{\mathbbm{i}}(y)$ and
$\bar{G}_s^{\mathbbm{i}}\(a,y\)=\ind_{\{a\leq u_s^{\mathbbm{i}}\(y\)\}} - \ind_{\{a\leq {\tilde{u}}_s^{\mathbbm{i}}\(y\)\}}$.
Moreover, we denote
\begin{equation}\label{I124}
\begin{split}
&I_1^{m,k,\mathbbm{i}}=\frac{1}{2}\mathbbm{E}\left[\int_0^t\int_{\mathbbm{R}}\phi_k^{'}\(\< v_s^{\mathbbm{i}},\Phi_m\(x-\cdot\)\>\)
\< v_s^{\mathbbm{i}},\Delta_{y}\Phi_m\(x-\cdot\)\> J\(x\)dxds\right];\\
&I_2^{m,k,\mathbbm{i}}=\mathbbm{E}\left[\int_0^t\int_{\mathbbm{R}}\phi_k^{'}\(\< v_s^{\mathbbm{i}},\Phi_m\(x-\cdot\)\>\)
\< v_s^{\mathbbm{i}},\Phi_m\(x-\cdot\)\> J\(x\)dxds\right];\\
&{I}_3^{m,k,\mathbbm{i}}=\mathbbm{E}\left[\int_0^t\int_{\mathbbm{R}}\int_{\mathbbm{R}^{+}}\phi_k''\(\< v_s^{\mathbbm{i}},\Phi_m\(x-\cdot\)\>\)
\left|\int_{\mathbbm{R}}\bar{G}_s^{\mathbbm{i}}\(a,y\)\Phi_m\(x-y\)dy\right|^2 daJ\(x\)dxds\right].
\end{split}
\end{equation}

\begin{proposition}\label{prop1}
For $\mathbbm{i}=\mathbbm{1,2}$ we have
\begin{equation*}
\begin{split}
\mathbbm{E}\left[\int_{\mathbbm{R}}\phi_k\(\<v_t^{\mathbbm{i}},\Phi_m\(x-\cdot\)\>\)J\(x\)dx\right]=
I_{1}^{m,k,\mathbbm{i}}+b_\mathbbm{i}I_{2}^{m,k,\mathbbm{i}}+\frac{\gamma_\mathbbm{i}}{2}I_3^{m,k,\mathbbm{i}}+I_{4}^{m,k,\mathbbm{i}},
\end{split}
\end{equation*}
where $I_{1}^{m,k,\mathbbm{i}}, I_{2}^{m,k,\mathbbm{i}}$ and $I_3^{m,k,\mathbbm{i}}$ are given by \eqref{I124},
\begin{equation}\label{I_31}
{{I}}_4^{m,k,\mathbbm{1}}=-\mathbbm{E}\left[\int_0^t\int_{\mathbbm{R}}\int_{\mathbbm{R}}\phi^{'}_k\(\< v_s^{\mathbbm{1}},\Phi_m\(x-\cdot\)\>\)\Phi_m\(x-y\) \overline{\xi}_s(y) dyJ\(x\)dxds\right]
\end{equation}
and
\begin{equation}\label{I_32}
{{I}}_4^{m,k,\mathbbm{2}}=\mathbbm{E}\left[\int_0^t\int_{\mathbbm{R}}\phi_k'\(\< v_s^{\mathbbm{2}},\Phi_m\(x-\cdot\)\>\)\<\dot{\chi},\Phi_m\(x-\cdot\)\>
\overline{\xi}_s(\infty) J\(x\) dxds\right]
\end{equation}
with $\overline{\xi}_s(\cdot) = \xi\(\cdot,\mu_s^{\mathbbm{1}},\mu_s^{\mathbbm{2}}\)
-\xi\(\cdot,\tilde{\mu}_s^{\mathbbm{1}},\tilde{\mu}_s^{\mathbbm{2}}\).$
\end{proposition}

\proof
It follows from (\ref{eq:spdes}) that
\beqnn
v_t^{\mathbbm{1}}\(y\)\ar=\ar\sqrt{\gamma_{\mathbbm{1}}}\int_0^t\int_0^{\infty}\bar{G}_s^{\mathbbm{1}}\(a,y\)W^{\mathbbm{1}}\(ds\,da\)
+\int_0^t\(\frac{\Delta_{y}}{2}v_s^{\mathbbm{1}}\(y\)+b_{\mathbbm{1}}v_s^{\mathbbm{1}}\(y\)-\overline{\xi}_s(y) \)ds
\eeqnn
and
\beqnn
v_t^{\mathbbm{2}}\(y\)\ar=\ar\sqrt{\gamma_{\mathbbm{2}}}\int_0^t\int_0^{\infty}\bar{G}_s^{\mathbbm{2}}\(a,y\)W^{\mathbbm{2}}\(ds\,da\)
+\int_0^t\(\frac{\Delta_{y}}{2}v_s^{\mathbbm{2}}\(y\)+b_{\mathbbm{2}}v_s^{\mathbbm{2}}\(y\)+\dot{\chi}\(y\)\overline{\xi}_s(\infty)\)ds.
\eeqnn
Consequently, we have
\beqlb\label{v_1}
\langle v_t^{\mathbbm{1}},\Phi_m\(x-\cdot\)\rangle
\ar=\ar \sqrt{\gamma_{\mathbbm{1}}}\int_0^t\int_{\mathbbm{R}^{+}}\int_{\mathbbm{R}}\bar{G}_s^{\mathbbm{1}}\(a,y\)\Phi_m\(x-y\)dyW^{\mathbbm{1}}\(ds\,da\)\cr
\ar\ar + b_{\mathbbm{1}}\int_0^t\big\langle v_s^{\mathbbm{1}},\Phi_m\(x-\cdot\)\big\rangle ds -\int_0^t\int_{\mathbbm{R}}\Phi_m\(x-y\)\overline{\xi}_s(y)dyds\cr
\ar\ar + \frac{1}{2}\int_0^t\< v_s^{\mathbbm{1}},\Delta_{y}\Phi_m\(x-\cdot\)\> d s
\eeqlb
and
\beqlb\label{v_2}
\langle v_t^{\mathbbm{2}},\Phi_m\(x-\cdot\)\rangle\ar=\ar \sqrt{\gamma_{\mathbbm{2}}}\int_0^t\int_{\mathbbm{R}^{+}}
\int_{\mathbbm{R}}\bar{G}_s^{\mathbbm{2}}\(a,y\)\Phi_m\(x-y\)dyW^{\mathbbm{2}}\(ds\,da\) \cr
\ar\ar +\frac{1}{2}\int_0^t \< v_s^{\mathbbm{2}},\Delta_{y}\Phi_m\(x-\cdot\)\> ds+b_{\mathbbm{2}}\int_0^t\big\langle v_s^{\mathbbm{2}},\Phi_m\(x-\cdot\)\big\rangle ds\cr
\ar\ar +\int_0^t\big\langle\dot{\chi},\Phi_m\(x-\cdot\)\big\rangle\overline{\xi}_s(\infty) ds.
\eeqlb
 Applying It\^o's formula to (\ref{v_1}) and (\ref{v_2}), we can easily get
\beqnn
\ar\ar\phi_k\(\langle v_t^{\mathbbm{1}},\Phi_m\(x-\cdot\)\rangle\)\cr
\ar\ar\qquad=\,\,\sqrt{\gamma_{\mathbbm{1}}}\int_0^t\int_{\mathbbm{R}^{+}}\phi_k'\(\langle v_s^{\mathbbm{1}},\Phi_m\(x-\cdot\)\rangle\)\int_{\mathbbm{R}}\bar{G}_s^{\mathbbm{1}}\(a,y\)\Phi_m\(x-y\)dyW^{\mathbbm{1}}\(ds\,da\)\\
\ar\ar\qquad\qquad +\int_0^t\phi_k'\(\langle v_s^{\mathbbm{1}},\Phi_m\(x-\cdot\)\rangle\)\Big[{\frac12\< v_s^{\mathbbm{1}}, \Delta_{y}\Phi_m\(x-\cdot\)\> } +b_{\mathbbm{1}} \big\langle v_s^{\mathbbm{1}},\Phi_m\(x-\cdot\)\big\rangle\Big] ds\\
\ar\ar\qquad\qquad+\frac{\gamma_{\mathbbm{1}}}{2}\int_0^t\int_{\mathbbm{R}^{+}}\phi''_k\(\langle v_s^{\mathbbm{1}},\Phi_m\(x-\cdot\)\rangle\)
\left|\int_{\mathbbm{R}}\bar{G}_s^{\mathbbm{1}}\(a,y\)\Phi_m\(x-y\)dy\right|^2dads\\
\ar\ar\qquad\qquad-\int_0^t\int_{\mathbbm{R}}\phi'_k\(\langle v_s^{\mathbbm{1}},\Phi_m\(x-\cdot\)\rangle\)\Phi_m\(x-y\) \overline{\xi}_s(y) dyds.
\eeqnn
and
\beqnn
\ar\ar\phi_k\(\langle v_t^{\mathbbm{2}},\Phi_m\(x-\cdot\)\rangle\)\cr
\ar\ar\qquad=\sqrt{\gamma_{\mathbbm{2}}}\int_0^t\int_{\mathbbm{R}^{+}}\phi_k'\(\langle v_s^{\mathbbm{2}},\Phi_m\(x-\cdot\)\rangle\)\int_{\mathbbm{R}}\bar{G}_s^{\mathbbm{2}}\(a,y\)\Phi_m\(x-y\)dyW^{\mathbbm{2}}\(ds\,da\)\\
\ar\ar\qquad\qquad+\int_0^t\phi_k'\(\langle v_s^{\mathbbm{2}},\Phi_m\(x-\cdot\)\rangle\)\Big[{ \frac12\< v_s^{\mathbbm{2}}, \Delta_{y}\Phi_m\(x-\cdot\)\>} +b_{\mathbbm{2}}\big\langle v_s^{\mathbbm{2}},\Phi_m\(x-\cdot\)\big\rangle\Big] ds\\
\ar\ar\qquad\qquad+\frac{\gamma_{\mathbbm{2}}}{2}\int_0^t\int_{\mathbbm{R}^{+}}\phi''_k\(\langle v_s^{\mathbbm{2}},\Phi_m\(x-\cdot\)\rangle\)
\left|\int_{\mathbbm{R}}\bar{G}_s^{\mathbbm{2}}\(a,y\)\Phi_m\(x-y\)dy\right|^2dads\\
\ar\ar\qquad\qquad+\int_0^t\phi_k'\(\langle v_s^{\mathbbm{2}},\Phi_m\(x-\cdot\)\rangle\)\langle\dot{\chi},\Phi_m\(x-\cdot\)\rangle
\overline{\xi}_s(\infty) ds.
\eeqnn
Taking the expectations of $\big\langle\phi_k\(\langle v_t^{\mathbbm{i}},\Phi_m\(x-\cdot\)\rangle\),\,J\(x\)\big\rangle_1$
with $\mathbbm{i}=\mathbbm{1,2}$, we obtain the desired results.
\qed

\begin{lemma}\label{l4.2}
For $\mathbbm{i}=\mathbbm{1,2}$ we have
\beqnn
2I_1^{m,k,\mathbbm{i}}\leq \mathbbm{E}\left[\int_0^t\int_{\mathbbm{R}}\< |v_s^{\mathbbm{i}}|,\Phi_m\(x-\cdot\)\> |J^{''}\(x\)|dxds\right].
\eeqnn
\end{lemma}

\proof
Note that
\beqlb\label{eq:I_1}
2I_1^{m,k,\mathbbm{i}}\ar=\ar\mathbbm{E}\[\int_0^t\int_{\mathbbm{R}}\phi_k'\(\< v_s^{\mathbbm{i}},\Phi_m\(x-\cdot\)\>\)\< v_s^{\mathbbm{i}},\Delta_{y}\Phi_m\(x-\cdot\)\> J\(x\)dx ds\]\cr
\ar=\ar\mathbbm{E}\[\int_0^t\int_{\mathbbm{R}}\phi_k'\(\< v_s^{\mathbbm{i}},\Phi_m\(x-\cdot\)\>\)\Delta_{x}\< v_s^{\mathbbm{i}},\Phi_m\(x-\cdot\)\> J\(x\)dx ds\]\cr
\ar=\ar-\mathbbm{E}\[\int_0^t\int_{\mathbbm{R}}\(\frac{\partial}{\partial x}
\< v_s^{\mathbbm{i}},\Phi_m\(x-\cdot\)\>\)^2\phi_k''\(\< v_s^{\mathbbm{i}},\Phi_m\(x-\cdot\)\>\) J\(x\)dxds\]\cr
\ar\ar-\mathbbm{E}\[\int_0^t\int_{\mathbbm{R}}\frac{\partial}{\partial x}\< v_s^{\mathbbm{i}},\Phi_m\(x-\cdot\)\> \phi_k'\(\< v_s^{\mathbbm{i}},\Phi_m\(x-\cdot\)\>\) J'\(x\)dxds\]\cr
\ar\leq\ar-\mathbbm{E}\[\int_0^t\int_{\mathbbm{R}}\frac{\partial}{\partial x}\< v_s^{\mathbbm{i}},\Phi_m\(x-\cdot\)\> \phi_k'\(\< v_s^{\mathbbm{i}},\Phi_m\(x-\cdot\)\>\) J'\(x\)dxds\]\cr
\ar=\ar-\mathbbm{E}\[\int_0^t\int_{\mathbbm{R}}\frac{\partial}{\partial x}\( \phi_k\(\<v_s^{\mathbbm{i}},\Phi_m\(x-\cdot\)\>\)\)
J'\(x\)dxds\]\cr
\ar=\ar\mathbbm{E}\[\int_0^t\int_{\mathbbm{R}}\phi_k\(\< v_s^{\mathbbm{i}},\Phi_m\(x-\cdot\)\>\)J''\(x\)dx ds\].
\eeqlb
Use $\phi_k\(z\)\leq |z|$ to get
$$\phi_k\(\< v_s^{\mathbbm{i}},\Phi_m\(x-\cdot\)\>\)\leq \Big|\< v_s^{\mathbbm{i}},\Phi_m\(x-\cdot\)\>\Big|\leq \< |v_s^{\mathbbm{i}}|,\Phi_m\(x-\cdot\)\>.$$
That implies the result.
\qed

\begin{lemma}\label{l4.3}
For $\mathbbm{i}=\mathbbm{1,2}$ we have
\beqnn
I_3^{m,k,\mathbbm{i}} \le 4\mathbbm{E}\[\int_0^t\int_{\mathbbm{R}}\phi_k^{''}\(\< v_s^{\mathbbm{i}},\Phi_m\(x-\cdot\)\>\)\<
|v_s^{\mathbbm{i}}|,\Phi_m\(x-\cdot\)\>J\(x\)dxds\].
\eeqnn
\end{lemma}
\proof
It is easy to see that
\beqnn
I_3^{m,k,\mathbbm{i}}\ar\leq\ar\mathbbm{E}\[\int_0^t\int_{\mathbbm{R}}\int_{\mathbbm{R}^{+}}\phi_k^{''}\(\< v_s^{\mathbbm{i}},\Phi_m\(x-\cdot\)\>\)
\int_{\mathbbm{R}}\(\bar{G}_s^{\mathbbm{i}}\(a,y\)\)^2\Phi_m\(x-y\)dy daJ\(x\)dxds\]\cr
\ar\leq\ar 4\mathbbm{E}\[\int_0^t\int_{\mathbbm{R}}\phi_k^{''}\(\< v_s^{\mathbbm{i}},\Phi_m\(x-\cdot\)\>\)\int_{\mathbbm{R}}
\left|v_s^{\mathbbm{i}}\(y\)\right|\Phi_m\(x-y\)dyJ\(x\)dxds\].
\eeqnn
The result follows.
\qed

\begin{theorem}\label{th:uniqueness}
Assume that there exists a constant $K$ such that
$$\big|{\xi}\(x,\nu_1,\nu_2\)-{\xi}\(x,\tilde{\nu}_1,\tilde{\nu}_2\)\big|\leq K\big[\rho\(\nu_1,\tilde{\nu}_1\)+\rho\(\nu_2,\tilde{\nu}_2\)\big]$$
for any $x\in\mathbbm{R}\cup\{\pm\infty\}$ and $\nu_i,\tilde{\nu}_i \in M_{F}\(\mathbbm{R}\)$ with $i = 1, 2$. Then the pathwise uniqueness holds for SPDEs (\ref{eq:spdes}), namely, if (\ref{eq:spdes}) has two solutions defined on the same stochastic basis with  the same initial values, then the solutions coincide almost surely.
\end{theorem}
\proof
The coincidence of $\left(u_t^{\mathbbm{1}},u_t^{\mathbbm{2}}\right)_{t\geq0}$ and $\left({\tilde{u}}_t^{\mathbbm{1}},{\tilde{u}}_t^{\mathbbm{2}}\right)_{t\geq0}$ 
is sufficient to the pathwise uniqueness of SPDEs (\ref{eq:spdes}). 
Subsequently, we estimate the values of $I_{\ell}^{m,k,\mathbbm{i}}$ with $\ell=1,2,3,4$ and $\mathbbm{i} = \mathbbm{1},\mathbbm{2}.$ 
Since
\begin{equation*}
\lim_{m\rightarrow\infty}\< v_s^{\mathbbm{i}},\Phi_m\(x-\cdot\)\> =v_s^{\mathbbm{i}}\(x\)\qquad\text{and}\qquad \lim_{m\rightarrow\infty}\< |v_s^{\mathbbm{i}}|,\Phi_m\(x-\cdot\)\>=|v_s^\mathbbm{i}\(x\)|
\end{equation*}
for Lebesgue-a.e. $x$ and any $s\geq 0$ almost surely.  By Lemma~\ref{l4.2} and dominated convergence theorem, we have
\begin{equation*}
\limsup_{k,m\rightarrow\infty}2I_1^{m,k,\mathbbm{i}}\leq \mathbbm{E}\[\int_0^t\int_{\mathbbm{R}}|v_s^{\mathbbm{i}}\(x\)|\cdot |J^{''}\(x\)|dx ds\].
\end{equation*}
By (\ref{eq:J}), there exists a constant $K$ such that
\begin{equation}\label{eq:I_1}
\limsup_{k,m\rightarrow\infty}2I_1^{m,k,\mathbbm{i}}\leq K\mathbbm{E}\[\int_0^t\int_{\mathbbm{R}}|v_s^{\mathbbm{i}}\(x\)|\cdot |J\(x\)|dx ds\].
\end{equation}
Using $|\phi_k^{'}\(z\)|\leq 1$ and dominated convergence theorem, we can easily get
\begin{equation}\label{eq:I_2}
\limsup_{k,m\rightarrow\infty}\left|I_2^{m,k,\mathbbm{i}}\right|\leq \mathbbm{E}\[\int_0^t\int_{\mathbbm{R}}|v_s^{\mathbbm{i}}\(x\)|\cdot |J\(x\)|dx ds\].
\end{equation}
Recall that $\phi_k^{''}\(z\)\left|z\right|\leq 2k^{-1},$ by Lemma~\ref{l4.3} one shall check that
\begin{equation}\label{eq:I_4}
\begin{split}
\limsup_{m\rightarrow\infty}I_3^{m,k,\mathbbm{i}}\leq&K\mathbbm{E}\[\int_0^t\int_{\mathbbm{R}}\phi_k^{''}\big(v_s^{\mathbbm{i}}\(x\)\big)
\big|v_s^{\mathbbm{i}}\(x\)\big|J\(x\)dxds\]=O\(k^{-1}\).
\end{split}
\end{equation}
Recall that $\chi$ is a finite measure on $\mathbb{R}$ and $|\phi_k^{'}\(z\)|\leq 1$. By \eqref{I_31},\,\eqref{I_32} we have
\begin{equation}\label{eq:I_3}
\begin{split}
\limsup_{k,m\rightarrow\infty}|I_4^{m,k,\mathbbm{i}}|\leq&\,\, \mathbbm{E}\[\int_0^t\int_{\mathbbm{R}}
\big[|\overline{\xi}_s(x)| + |\overline{\xi}_s(\infty)|\big]\cdot |J\(x\)|dx ds\]\\
\leq&\,\,K\int_0^t\mathbbm{E}\[\big[\rho\({\mu}_s^{\mathbbm{1}},\tilde{\mu}_s^{\mathbbm{1}}\)+\rho\({\mu}_s^{\mathbbm{2}},\tilde{\mu}_s^{\mathbbm{2}}\)\big]ds\]\\
\leq &\,\,K\mathbbm{E}\[\int_0^t\int_{\mathbbm{R}}\(|v_s^{\mathbbm{1}}\(x\)|+|v_s^{\mathbbm{2}}\(x\)|\)J\(x\)dxds\].
\end{split}
\end{equation}
By Proposition~\ref{prop1} and putting \eqref{eq:I_1}-\eqref{eq:I_3} together, one can see that
\begin{equation*}
\mathbbm{E}\[\int_{\mathbbm{R}}\(\left|v_t^{\mathbbm{1}}\(x\)\right|+\left|v_t^{\mathbbm{2}}\(x\)\right|\)J\(x\)dx\]\leq
K\mathbbm{E}\[\int_0^t\int_{\mathbbm{R}}\(\left|v_s^{\mathbbm{1}}\(x\)\right|+\left|v_s^{\mathbbm{2}}\(x\)\right|\)J\(x\)dx ds\].
\end{equation*}
Then Gronwall's inequality implies that
\beqnn
\mathbbm{E}\[\int_{\mathbbm{R}}\(\left|v_t^{\mathbbm{1}}\(x\)\right|+\left|v_t^{\mathbbm{2}}\(x\)\right|\)J\(x\)dx\]=0
\eeqnn
 for any $t \ge 0$
and the pathwise uniqueness follows.
\qed



\begin{thebibliography}{99}


\bibitem[Dawson(1975)]{D75} D. A. Dawson.
\newblock Stochastic evolution equations and related measure processes.
\emph{Journal of Multivariate Analysis}, 5: 1--52, 1975.



\bibitem[Dawson(1993)]{Dawson}
D.~A. Dawson.
\newblock \emph{Measure-valued Markov Processes}.
\newblock Springer, Berlin, 1993.
\newblock Ecole dete de Probabilites de Sanit-Flour XIII-1991, Lecture Notes in
  Mathematics 1541.

\bibitem[Dawson and Li(2012)]{DL2012AP}
D.~A. Dawson and Z.~Li.
\newblock Stochastic equations, flows and measure-valued processes.
\newblock \emph{The Annals of Probability}, 40\penalty0 (2):\penalty0 813--857,
  2012.

\bibitem[Dawson and Perkins(1998)]{DP1998AP}
D.~A. Dawson and E.~Perkins.
\newblock Long-time behavior and coexistence in a mutually catalytic branching
  model.
\newblock \emph{The Annals of Probability}, 26\penalty0 (3):\penalty0
  1088--1138, 1998.

\bibitem[Dawson et~al.(2002{\natexlab{a}})Dawson, Etheridge, Fleischmann,
  Mytnik, Perkins, and Xiong]{D2002AP}
D.~A. Dawson, A.~Etheridge, K.~Fleischmann, L.~Mytnik, E.~Perkins, and
  J.~Xiong.
\newblock Mutually catalytic branching in the plane: finite measure states.
\newblock \emph{The Annals of Probability}, 30\penalty0 (4):\penalty0
  1681--1762, 2002{\natexlab{a}}.

\bibitem[Dawson et~al.(2002{\natexlab{b}})Dawson, Etheridge, Fleischmann,
  Mytnik, Perkins, and Xiong]{D2002EJP}
D.~A. Dawson, A.~Etheridge, K.~Fleischmann, L.~Mytnik, E.~Perkins, and
  J.~Xiong.
\newblock Mutually catalytic branching in the plane: infinite measure states.
\newblock \emph{Electronic Journal of Probability}, 7\penalty0 (15):\penalty0
  1--61, 2002{\natexlab{b}}.

\bibitem[Etheridge(2000)]{Etheridge}
A.~Etheridge.
\newblock \emph{An introduction to superprocesses}, volume~20.
\newblock University Lecture Series, 2000.

\bibitem[He et~al.(2014)He, Li, and Yang]{HLY}
H.~He, Z.~Li, and X.~Yang.
\newblock Stochastic equations of super-{L}\'{e}vy processes with general
  branching mechanism.
\newblock \emph{Stochastic Processes and their Applications}, 124:\penalty0
  1519--1565, 2014.

\bibitem[Jacod and Shiryaev(2003)]{JS03}
J.~Jacod and A.~N. Shiryaev.
\newblock \emph{Limit theorems for stochastic processes. 2nd Ed. Grundlehren
  der Mathematischen Wissenschaften}, volume 288.
\newblock Springer-Verlag, Berlin, 2003.

\bibitem[Karoui and M\`{e}l\`{e}ard(1990)]{EM90}
N.~E. Karoui and S.~M\`{e}l\`{e}ard.
\newblock Martingale measures and stochastic calculus.
\newblock \emph{Probability Theory and Related Fields}, 84:\penalty0 83--101,
  1990.

\bibitem[Li(1992)]{Li1992}
Z.~Li.
\newblock Measure-valued branching processes with immigration.
\newblock \emph{Stochastic Processes and their Applications}, 43\penalty0
  (2):\penalty0 249--264, 1992.

\bibitem[Li(1996)]{Li1996}
Z.~Li.
\newblock Immigration structures associated with {D}awson-{W}atanabe
  superprocesses.
\newblock \emph{Stochastic Processes and their Applications}, 62\penalty0
  (1):\penalty0 73--86, 1996.

\bibitem[Li(2011)]{Li2011}
Z.~Li.
\newblock \emph{Measure-Valued Branching Markov Processes}.
\newblock Springer-Verlag Berlin Heidelberg, 2011.

\bibitem[Li and Shiga(1995)]{LS1995}
Z.~Li and T.~Shiga.
\newblock Measure-valued branching diffusions: immigrations, excursions and
  limit theorems.
\newblock \emph{Kyoto Journal of Mathematics}, 35\penalty0 (2):\penalty0
  233--274, 1995.


\bibitem[Li et~al.(2010)]{LXZ2010}
Z.~Li, J.~Xiong, and M.~Zhang.
\newblock Ergodic theory for a superprocess over a stochastic flow. 
\newblock \emph{Stochastic
Processes and their Applications}, 120 (8):\penalty0
  1563--1588, 2010.



\bibitem[M\'{e}l\'{e}ard(1995)]{M96}
S.~M\'{e}l\'{e}ard.
\newblock \emph{Asymptotic behaviour of some interacting particle systems;
  {M}cKean-Vlasov and {B}oltzmann models, Probabilistic models for nonlinear
  partial differential equations}, volume 1627.
\newblock Lecture Notes in Mathematics, 1995.

\bibitem[M\'{e}l\'{e}ard and Roelly(1992)]{M92}
S.~M\'{e}l\'{e}ard and S.~Roelly.
\newblock Interacting measure branching processes. some bounds for the support.
\newblock \emph{Stochastics and Stochastics Reports}, 44:\penalty0 103--121,
  1992.

\bibitem[Mytnik(1998)]{M1998}
L.~Mytnik.
\newblock Uniqueness for a mutually catalytic branching model.
\newblock \emph{Probability Theory and Related Fields}, 112:\penalty0 245--253,
  1998.

\bibitem[Mytnik and Xiong(2015)]{MX2015}
L.~Mytnik and J.~Xiong.
\newblock Well-posedness of the martingale problem for superprocess with
  interaction.
\newblock \emph{Illinois Journal of Mathematcis}, 59\penalty0 (2):\penalty0
  485--497, 2015.

\bibitem[Xiong(2008)]{filter}
J.~Xiong.
\newblock \emph{An Introduction to Stochastic Filtering Theory}.
\newblock Oxford Graduate Texts in Mathematics Inc., New York, 2008.

\bibitem[Xiong(2013)]{X2013}
J.~Xiong.
\newblock Super-{B}rownian motion as the unique strong solution to an spde.
\newblock \emph{The Annals of Probability}, 41\penalty0 (2):\penalty0
  1030--1054, 2013.

\bibitem[Xiong and Yang(2016)]{XY2016}
J.~Xiong and X.~Yang.
\newblock Superprocesses with interaction and immigration.
\newblock \emph{Stochastic Processes and their Applications}, 126:\penalty0
  3377--3401, 2016.

\bibitem[Watanabe(1968)]{W68} S. Watanabe.
\newblock A limit theorem of branching processes and continuous state branching processes.
\newblock\emph{Journal of Mathematics of Kyoto University}, 8: 141--167, 1968.




\end{thebibliography}
\end{document}